\documentclass[11pt]{amsart}

\usepackage{amsmath, amsthm, amssymb}

\usepackage{palatino}         
\usepackage[abs]{overpic}		  

\usepackage{xcolor}
\definecolor{indigo}{rgb}{0.29, 0.0, 0.51}  
\definecolor{dred}{RGB}{237, 28, 36}
\usepackage[colorlinks, urlcolor=indigo, linkcolor=indigo, citecolor=indigo]{hyperref}
\usepackage[normalem]{ulem}
\theoremstyle{plain}
\newtheorem{theorem}{Theorem}
\newtheorem{corollary}[theorem]{Corollary}
\newtheorem{proposition}[theorem]{Proposition}
\newtheorem{lemma}[theorem]{Lemma}
\newtheorem{question}[theorem]{Question}
\newtheorem{conjecture}[theorem]{Conjecture}

\theoremstyle{definition}
\newtheorem{definition}[theorem]{Definition}

\theoremstyle{remark}
\newtheorem{remark}[theorem]{Remark}
\newtheorem{example}[theorem]{Example}

\numberwithin{theorem}{section}

\newcommand{\dfn}[1]{{\em #1}}        
\newcommand{\R}{\mathbb{R}}           
\newcommand{\Z}{\mathbb{Z}}           
\newcommand{\Q}{\mathbb{Q}} 
\newcommand{\C}{\mathbb{C}}           

\makeatletter
\newcommand*\bigcdot{\mathpalette\bigcdot@{0.6}}
\newcommand*\bigcdot@[2]{\mathbin{\vcenter{\hbox{\scalebox{#2}{$\m@th#1\bullet$}}}}}
\makeatother

\DeclareMathOperator\tb{tb}                   

\begin{document}

\title{Symplectic rational homology ball fillings \linebreak of Seifert fibered spaces}

\author{John B. Etnyre}

\author{Burak Ozbagci}

\author{B\"{u}lent Tosun}

\address{Department of Mathematics \\ Georgia Institute of Technology \\ Atlanta \\ Georgia}

\email{etnyre@math.gatech.edu}

\address{Department of Mathematics \\ Ko\c{c} University \\ Istanbul \\ Turkey}

\email{bozbagci@ku.edu.tr}

\address{Institute for Advanced Study at Princeton and Department of Mathematics\\ University of Alabama\\Tuscaloosa\\Alabama}

\email{btosun@ua.edu}


\begin{abstract}
We characterize when some small Seifert fibered spaces can be the convex boundaries of symplectic rational homology balls and give strong restrictions for others to bound such manifolds. In particular, we show that the only spherical $3$-manifolds, oriented as  links of the corresponding quotient singularities, which admit symplectic rational homology ball fillings, are the lens spaces $L(p^2,pq-1)$ previously identified by Lisca. In a different  direction, we provide evidence for  Gompf's conjecture that Brieskorn spheres do not bound Stein domains in $\C^2$. Finally, we establish restrictions on Lagrangian disk fillings of certain Legendrian knots in small Seifert fibered spaces.
\end{abstract}

\maketitle


\section{Introduction}

There has been considerable recent interest in determining which rational homology $3$-spheres bound rational homology $4$-balls \cite{AcetoGolla2017, AcetoGollaLecuona2018, BhupalStipsicz11, ChoePark2021, Lecuona2018, Lecuona2019, Lisca2007,  Simone2021}. This problem is interesting in its own right, but it also plays a key role in the study of generalized rational blowdowns \cite{FintushelStern1997, ParkStipsicz2014, Symington2001}. In particular, if one wishes to perform such a blowdown in the symplectic category, it becomes essential to understand when a rational homology 3-sphere equipped with a contact structure admits a symplectic rational homology ball  filling. 

This article identifies contact structures on certain small Seifert fibered spaces that admit symplectic rational homology ball fillings, and establishes strong obstructions in most other cases. We also point out a striking difference between the smooth and the symplectic categories: a small Seifert fibered space may bound a rational homology ball {\em smoothly} but fail to do so  {\em symplectically}.

As a corollary of our results and their method of proof, we will prove a special case of a conjecture due to Gompf: 

\begin{conjecture}[Gompf 2013, \cite{Gompf13}]\label{con: gompf}
 No Brieskorn integer homology sphere (other than $S^3$) admits a pseudoconvex embedding in $\mathbb{C}^2$, with either orientation.
\end{conjecture}

Note that, if an integer homology $3$-sphere has a pseudoconvex embedding in $\mathbb{C}^2$, then the compact region it bounds is necessarily a Stein integral homology ball. In connection with Gompf's conjecture, we show that for $(p,q)\not=(2,3)$,  the Brieskorn homology sphere $\Sigma(p,q, pqn+1)$ (see Section~\ref{brieskornS} for precise definitions and conventions) 
does not even bound a symplectic rational homology ball.

As another corollary, we show that on a spherical $3$-manifold, which is oriented as the link of the corresponding quotient singularity, only the canonical structure can admit a symplectic rational homology ball filling. To place this in context, we begin by recalling the well-studied case of lens spaces. Lisca \cite{Lisca2007} characterized the lens spaces that smoothly bound rational homology balls. He showed that the lens space $L(p,q)$ bounds a rational homology ball if and only if $p/q$ belongs to a certain subset $\mathcal{R}$ of positive rationals $\Q_+$. Since an actual definition of  $\mathcal{R}$ will not be needed here, we refer the reader to \cite{Lisca2007} for details. 

Each lens space $L(p,q)$ carries a unique universally tight contact structure $\xi_{can}$,  up to isomorphism. This isomorphism class consists of two non-isotopic contact structures, $\pm \xi_{can}$, except in the special case $q=p-1$, where $\xi_{can}$ is even isotopic to $-\xi_{can}$. Define $\mathcal{O} \subset \Q_+$ by 
\begin{equation}\label{eq: O}
\mathcal{O}=\left\{ \dfrac{m^2}{mh-1} \; | \; h < m \;  \mbox{are coprime positive integers} \right\}.   
\end{equation} 

By construction, $\mathcal{O} \subset \mathcal{R}$. Lisca \cite{Lisca08} proved that $(L(p,q), \xi_{can})$ admits a symplectic rational homology ball filling if and only if $p/q \in \mathcal{O}$, and that such a filling is unique up to diffeomorphism. In \cite{BhupalOno2012}, Bhupal and Ono enhanced this result to prove uniqueness up to symplectic deformation equivalence.

Moreover, Golla and Starkston \cite[Proposition A.1]{GollaStarkston2022} showed that the lens space $L(p,q)$ equipped with any tight contact structure other than $\xi_{can}$, does not admit a symplectic  rational homology ball filling. Alternative independent proofs of this fact were also given by Roy and the first author \cite[Lemma 1.6]{EtnyreRoy21}, Christian and Li \cite{ChristianLi23}, and by the first and third authors \cite[Proposition 11]{EtnyreTosun2023}. A partial case restricted to  $p/q \in \mathcal{O}$ was established earlier by Fossati \cite[Theorem 4]{Fossati2020}.  

Combining these results, we obtain the following complete classification: $(L(p,q), \xi)$ admits a  symplectic rational homology ball  filling if and only if $p/q \in \mathcal{O}$, and $\xi$ is contactomorphic to $\xi_{can}$. Moreover,  for any $p/q \in \mathcal{O}$, the contact $3$-manifold $(L(p,q), \xi_{can})$ admits a unique symplectic rational homology ball filling, up to symplectic deformation.

Since every lens space is a spherical (also known as an elliptic) 3-manifold, this raises the natural question of whether the classification above extends to all spherical 3-manifolds.  As a corollary of our main results, we prove that this is indeed the case, provided that the spherical 3-manifolds are canonically oriented as links of the corresponding quotient singularities. 

We emphasize, however,  that this result does not necessarily hold true for a spherical $3$-manifold equipped with the orientation opposite to the canonical one it carries when viewed as the singularity link (see Remark~\ref{rem: notTrue} below). In  fact, in subsequent work  \cite{EtnyreOzbagciTosun25pre}, we provide a complete classification of spherical $3$-manifolds {\em with either orientations}, which admit  symplectic rational homology ball fillings.

\subsection{Fillings of small Seifert fibered spaces}\label{subsec: fillsfs}

For $e_0\in \Z$ and $r_i\in(0,1) \cap \mathbb{Q}$,  
the small Seifert fibered space $Y(e_0; r_1, r_2, r_3)$ is described  by the surgery diagram depicted in Figure~\ref{fig:small}. 
\begin{figure}[htb]{
\begin{overpic}
{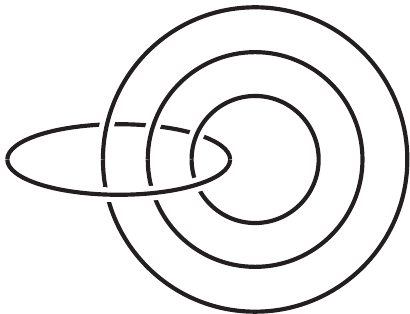}
\put(10, 55){$e_0$}
\put(140, 100){$-\frac 1{r_1}$}
\put(159, 113){$-\frac 1{r_2}$}
\put(179, 125){$-\frac 1{r_3}$}
\end{overpic}}
 \caption{A surgery diagram for small Seifert fibered space $Y(e_{0};r_{1},r_{2},r_{3})$.}
  \label{fig:small}
\end{figure}
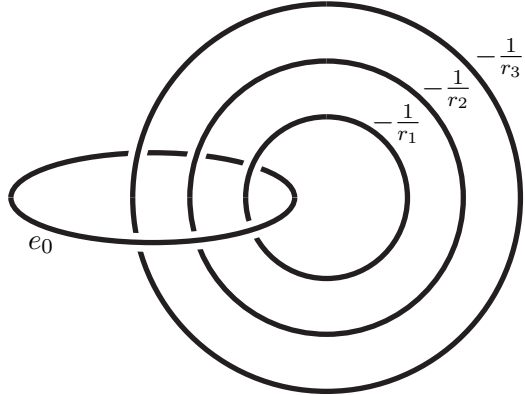
Recall that a {\em rational homology ball} is a 4-manifold $X$ such that $H_{*}(X; \mathbb{Q})=H_{*}(B^4; \mathbb{Q})$. Note that the boundary of a rational homology ball is a rational homology sphere, and we mention that $Y(e_0;r_1,r_2,r_3)$ is a rational homology sphere if and only if $e_0+r_1+r_2+r_3\not=0$. In particular, if $e_0\not=-1$ or $-2$, then $Y(e_0;r_1,r_2,r_3)$ is automatically a rational homology sphere, and if $e_0=-1$ or $-2$, then $Y(e_0;r_1,r_2,r_3)$  is still  a rational homology sphere,  except in the special cases where $r_1+r_2+r_3=1$ or $2$.  

Since the symplectic fillability of a closed contact $3$-manifold depends on the underlying tight contact structure, we begin by recalling what is known about the classification of  tight contact structures on small Seifert fibered spaces. 

Wu \cite{Wu06} proved that when $e_0<-2$, any tight contact structure on $Y=Y(e_0;r_1,r_2,r_3)$ arises from Legendrian surgery on a suitable Legendrian realization of the surgery diagram in Figure~\ref{fig:plumbing}. Here the surgery coefficients $-a^i_j$ are determined by the continued fraction expansion \[
-1/r_i=[-a^i_0,-a^i_1,\ldots, -a^i_{n_i}]=-a^i_0-
\cfrac{1}{-a^i_1- \cfrac{1}{\ddots  -\cfrac{1}{-a^i_{n_i}}}}
\]  where $a_j^i\geq 2$. 
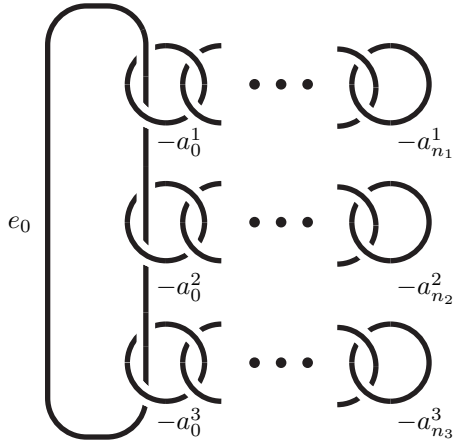
\begin{figure}[htb]{\small
\begin{overpic}
{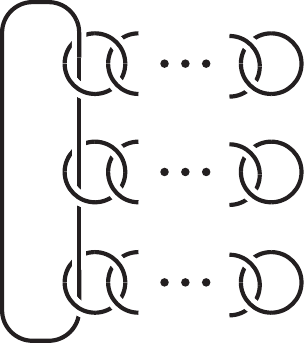}
\put(-14, 80){$e_0$}
\put(42, 110){$-a^1_0$}
\put(133, 110){$-a^1_{n_1}$}
\put(42, 55){$-a^2_0$}
\put(133, 55){$-a^2_{n_2}$}
\put(42, 4){$-a^3_0$}
\put(133, 4){$-a^3_{n_3}$}
\end{overpic}}
\caption{A surgery diagram for $Y(e_0;r_1,r_2,r_3)$.}
\label{fig:plumbing}
\end{figure}
By a {\em suitable Legendrian realization} we mean the following: each unknot in the surgery diagram with coefficient  $-a_j^i$ is realized as a Legendrian unknot with Thurston-Bennequin invariant $1-a^i_j$. In particular, such a Legendrian unknot can be obtained from the maximal Thurston–Bennequin unknot by stabilizing $a^i_j-2$ times. The central unknot with surgery coefficient $e_0$  can also be realized in this manner. Furthermore, Ghiggini \cite{Ghiggini08} (see also \cite{Tosun20}) showed that  the same construction applies when $e_0=-2$, provided that  $Y$ is an $L$-space.

\begin{definition}\label{m-consistent} 
Consider the surgery diagram in Figure~\ref{fig:plumbing} for $Y = Y(e_0; r_1,r_2,r_3)$. 
\begin{itemize}
\item Each chain of unknots linked to the central $e_0$--framed unknot is called a \emph{leg} of the diagram.
\item Given a Legendrian realization of the surgery diagram, we say that a leg is \emph{consistent} if all stabilizations of the unknots in that leg have the same sign; otherwise, the leg is \emph{inconsistent}.
 \item The surgery diagram itself is called \emph{consistent} if all stabilizations of all Legendrian unknots in the diagram have the same sign.
\item The diagram is called \emph{mostly consistent} if each leg is consistent, but the diagram as a whole is not.
\item  A \emph{consistent} (resp. \emph{mostly consistent}) contact structure on $Y$ is the one obtained by Legendrian surgery on a consistent (resp. mostly consistent) diagram.
 \end{itemize}
\end{definition}

The \dfn{canonical contact structure} $\xi_{can}$ on $Y$ is the one obtained by Legendrian surgery on a consistent diagram, which is well-defined up to isomorphism. There is also an algebro–geometric description of $\xi_{can}$. The surgery diagram in Figure~\ref{fig:plumbing} also describes a $4$-manifold $X$ with a negative definite intersection form, provided that $e_0 <-2$ or $e_0=-2$ and $Y$ is an $L$-space. Therefore,   the boundary of $X$ is the link of a normal surface singularity \cite{Grauert1962}, {\em cf} \cite[Page~333]{Neumann1981} or \cite[Page~111]{Nemethi2013}, and as such the link $Y(e_0;r_{1},r_{2},r_{3})$ carries a canonical contact structure, known as the {\it Milnor fillable} contact structure,  which is uniquely determined up to isomorphism \cite{CaubelNemethiPopescu-Pampu2006}. Moreover, according to \cite[Theorem~8.1]{BhupalOzbagci11}, the Milnor fillable contact structure on $Y$ coincides with the  canonical contact structure $\xi_{can}$. 

Suppose that the small Seifert fibered space $Y=Y(e_0;r_{1},r_{2},r_{3})$ is the oriented link of a normal surface singularity. Then, according to \cite[Theorem 1.4]{BhupalStipsicz11},   $(Y, \xi_{can})$ admits a  symplectic rational homology ball filling if and only if the minimal good resolution graph (a negative definite star-shaped tree with three legs)  of the singularity belongs to one of the ten infinite families shown in \cite[Figure 1]{BhupalStipsicz11}  (see also \cite{SSW08, ParkShinStipsicz13}). For the convenience of the reader, we have reproduced these infinite families of plumbing graphs in Figures~\ref{fig:plumbing1} and~\ref{fig:plumbing2} below.  Here, we denote by $\mathcal{QHB}$ the set of small Seifert fibered spaces described by these families. 

\begin{figure}[htb]{\tiny
\begin{overpic}
{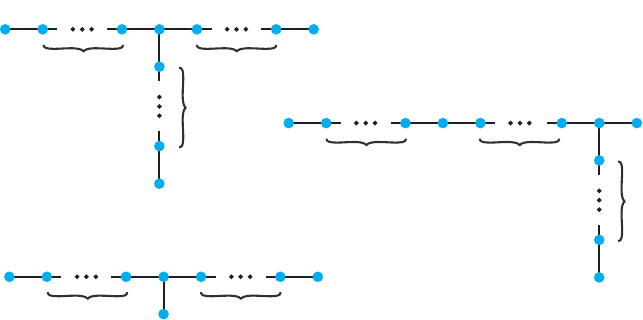}
\put(-16, 129){$-(p+3)$}
\put(13, 147){$-2$}
\put(38, 121){$q$}
\put(52, 147){$-2$}
\put(70, 147){$-4$}
\put(88, 147){$-2$}
\put(112, 121){$r$}
\put(125, 147){$-2$}
\put(136, 129){$-(q+3)$}
\put(62, 119){$-2$}
\put(94, 101){$p$}
\put(62, 81){$-2$}
\put(43, 63){$-(r+3)$}
\put(121, 84){$-(r+3)$}
\put(151, 102){$-2$}
\put(175, 76){$p$}
\put(187, 102){$-2$}
\put(206, 102){$-3$}
\put(224, 102){$-2$}
\put(248, 76){$q$}
\put(263, 102){$-2$}
\put(281, 102){$-3$}
\put(293, 84){$-(p+3)$}
\put(272, 74){$-2$}
\put(305, 56){$r$}
\put(272, 36){$-2$}
\put(253, 17){$-(q+4)$}
\put(-16, 10){$-(r+4)$}
\put(15, 28){$-2$}
\put(38, 2){$q$}
\put(52, 28){$-2$}
\put(70, 28){$-3$}
\put(88, 28){$-2$}
\put(112, 2){$r$}
\put(126, 28){$-2$}
\put(138, 10){$-(q+3)$}
\put(64, 0){$-2$}
\end{overpic}}
\caption{Plumbing diagrams in $\mathcal{QHB}$ for $Y(e_0;r_1,r_2,r_3)$ with $e_0\leq -3$.}
\label{fig:plumbing1}
\end{figure}
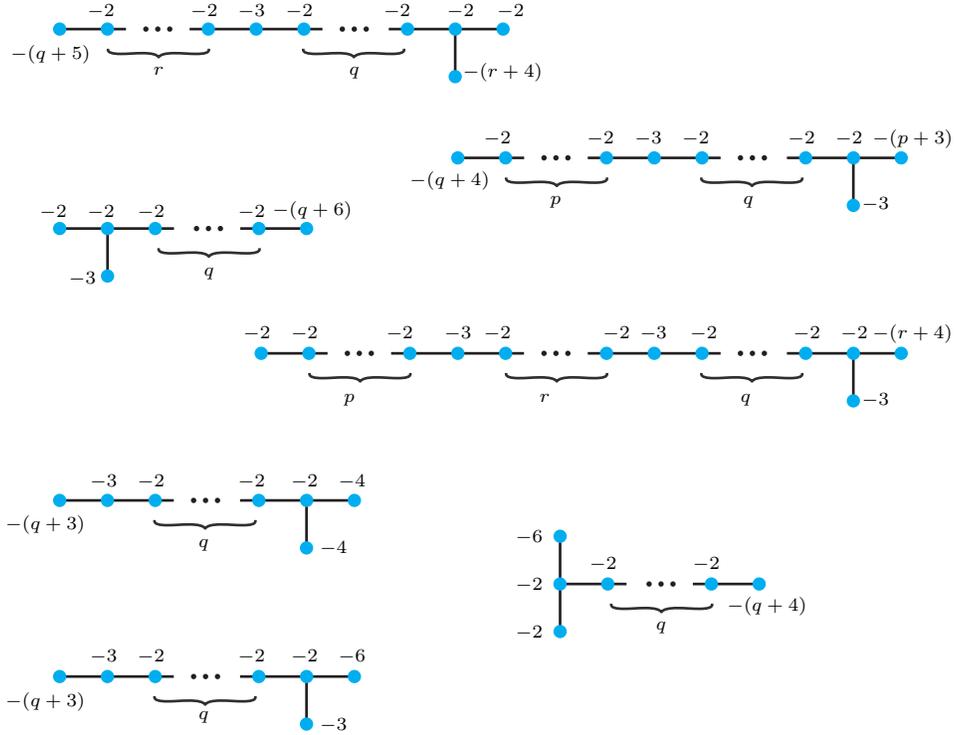
\begin{figure}[htb]{\tiny
\begin{overpic}
{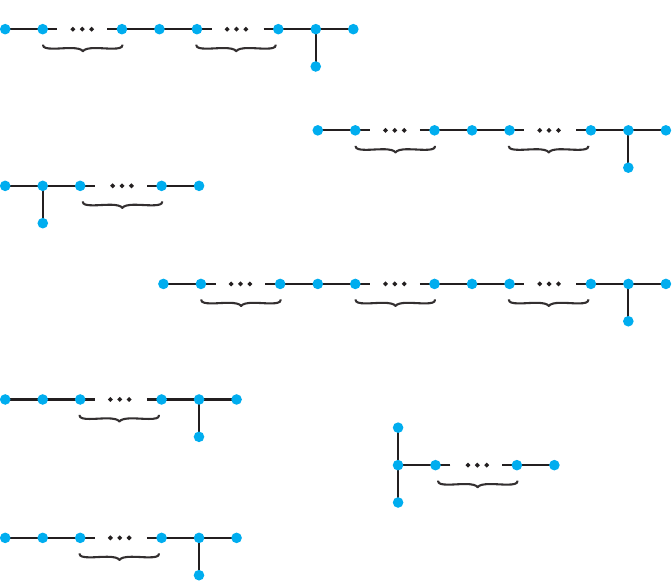}
\put(-16, 255){$-(q+5)$}
\put(13, 271){$-2$}
\put(38, 249){$r$}
\put(52, 271){$-2$}
\put(70, 271){$-3$}
\put(88, 271){$-2$}
\put(112, 249){$q$}
\put(125, 271){$-2$}
\put(149, 271){$-2$}
\put(168, 271){$-2$}
\put(155, 248){$-(r+4)$}
\put(135, 207){$-(q+4)$}
\put(163, 223){$-2$}
\put(188, 200){$p$}
\put(202, 223){$-2$}
\put(220, 223){$-3$}
\put(238, 223){$-2$}
\put(261, 200){$q$}
\put(278, 223){$-2$}
\put(296, 223){$-2$}
\put(310, 223){$-(p+3)$}
\put(306, 198){$-3$}
\put(-5, 195){$-2$}
\put(13, 195){$-2$}
\put(32, 195){$-2$}
\put(57, 173){$q$}
\put(70, 195){$-2$}
\put(83, 196){$-(q+6)$}
\put(6, 170){$-3$}
\put(72, 149){$-2$}
\put(90, 149){$-2$}
\put(110, 125){$p$}
\put(126, 149){$-2$}
\put(148, 149){$-3$}
\put(163, 149){$-2$}
\put(184, 125){$r$}
\put(208, 149){$-2$}
\put(222, 149){$-3$}
\put(241, 149){$-2$}
\put(260, 125){$q$}
\put(280, 149){$-2$}
\put(298, 149){$-2$}
\put(310, 149){$-(r+4)$}
\put(306, 124){$-3$}
\put(-18, 77){$-(q+3)$}
\put(14, 93){$-3$}
\put(32, 93){$-2$}
\put(55, 70){$q$}
\put(70, 93){$-2$}
\put(90, 93){$-2$}
\put(108, 93){$-4$}
\put(101, 68){$-4$}
\put(-18, 10){$-(q+3)$}
\put(14, 27){$-3$}
\put(32, 27){$-2$}
\put(55, 5){$q$}
\put(70, 27){$-2$}
\put(90, 27){$-2$}
\put(108, 27){$-6$}
\put(101, 1){$-3$}
\put(175, 72){$-6$}
\put(175, 54){$-2$}
\put(175, 36){$-2$}
\put(203, 62){$-2$}
\put(228, 39){$q$}
\put(242, 62){$-2$}
\put(255, 46){$-(q+4)$}
\end{overpic}}
\caption{Plumbing diagrams in $\mathcal{QHB}$ for $Y(-2;r_1,r_2,r_3)$.}
\label{fig:plumbing2}
\end{figure}

Our first two main results are as follows. 

\begin{theorem}\label{main1}
Suppose that $e_0\leq -3$. Then  $(Y(e_0;r_1,r_2,r_3), \xi)$ admits a symplectic rational homology ball filling if and only if  $Y(e_0;r_1,r_2,r_3) \in\mathcal{QHB}$ and $\xi$ is contactomorphic to the Milnor fillable contact structure $\xi_{can}$. \end{theorem}
We prove this theorem in Section~\ref{Mneg}.

\begin{remark} As a consequence of Theorem~\ref{main1}, we obtain the following:
\begin{itemize}
    \item If $e_0 < -4$, then there is no contact structure $\xi$ on $Y=Y(e_0; r_1,r_2,r_3)$ so that 
    $(Y, \xi)$ admits a symplectic rational homology ball filling.
    \item If $e_0 = -3$ or $e_0 = -4$, then only very few of the Seifert fibered spaces $Y=Y(e_0; r_1,r_2,r_3)$ carries contact structures $\xi$ so that  $(Y, \xi)$ admits a symplectic rational homology ball filling. In these cases, whenever such a contact structure exists, it is unique up to isomorphism and is universally tight.
\end{itemize}
\end{remark}
\begin{remark}
It is interesting to contrast this with the situation for {\em smooth} rational homology ball fillings. In the smooth category, such fillings can be constructed for small Seifert fibered spaces for any value of  $e_0 \in \Z$. See Remark~\ref{nege0smooth} below and also \cite{Aceto2020, Lecuona2018}. 
\end{remark}

We now turn to the case of small Seifert fibered spaces with $e_0=-2$. 

\begin{theorem}\label{main2}
If $Y=Y(-2;r_1,r_2,r_3)$ belongs to  $\mathcal{QHB}$, then $(Y, \xi)$ admits a symplectic rational homology ball filling if and only if $\xi$ is contactomorphic to $\xi_{can}$.  If $Y$ does not belong to $\mathcal{QHB}$ but it is an $L$-space, then $(Y, \xi)$ can admit a symplectic rational homology ball filling only if $\xi$ is consistent or mostly consistent. 
\end{theorem}

This theorem and its corollary below are proven in Section~\ref{Lnegm2}. At present, we do not know of any $Y \notin \mathcal{QHB}$ that admits a  symplectic rational homology ball filling. This motivates the following conjecture.

\begin{conjecture}
Suppose that $Y=Y(-2;r_1,r_2,r_3)$ is a small Seifert fibered space, which is also an $L$-space. If $(Y, \xi)$ admits a symplectic rational homology ball filling, then $Y \in \mathcal{QHB}$ and $\xi$ is contactomorphic to $\xi_{can}$. 
\end{conjecture}

While we do not know how to prove this conjecture yet, Theorem~\ref{main2} already places strong restrictions on the possible contact structures that can admit symplectic rational homology ball fillings. 

\begin{corollary}\label{possiblefillingsfore0m2}
If $Y=Y(-2;r_1,r_2,r_3)$ is an $L$-space, then, up to isomorphism,  there are at most four contact structures  $\xi$ on $Y$ so that $(Y, \xi)$ admits a  symplectic rational homology ball filling. 
\end{corollary}

Recall (see Lemma~\ref{lem: rhb}) that if a rational homology sphere with a contact structure $\xi$ can be symplectically filled by a rational homology ball, then $\theta(\xi)=-2$, where $\theta$ is Gompf's 3-dimensional homotopy invariant   \cite{Gompf98} (see Section~\ref{recalltheta} for precise definitions). The following proposition will be useful in our arguments and may also be of independent interest. 

\begin{proposition}\label{canminimizes}
Let $Y=Y(e_0;r_1,r_2,r_3)$ be a small Seifert fibered space with $e_0\leq -3$ or $e_0=-2$ and $Y$ an $L$-space. Then 
\[
\theta(\xi_{can})<\theta(\xi)
\]
for any contact structure $\xi$ on $Y$ not {isotopic} to $\pm \xi_{can}$. 
\end{proposition}

This proposition is proven in Section~\ref{sec: compare}. 

\begin{remark}\label{canminimizes1}
The proof of Proposition~\ref{canminimizes}  also applies to lens spaces. 
\end{remark}

We now turn to the case of small Seifert fibered spaces with $e_0\geq 0$. Results of Wu \cite{Wu06} for $e_0>0$ and Ghiggini, Lisca, and Stipsicz \cite{GhigginiLiscaStipsicz06} for $e_0\geq 0$ show that any tight contact structure on $Y(e_0;r_1,r_2,r_3)$ with $e_0\geq 0$ can be obtained by a contact surgery on the top diagram in Figure~\ref{fig:e0gqe0}. 
In this diagram, the coefficients $s_i$ are determined as follows. One performs Rolfsen twists on the curves in Figure~\ref{fig:small}, until the horizontal curve has framing $0$.  With further Rolfsen twists, if necessary, one may assume not only that the horizontal curve remains framed $0$ but also the inequalities  $s_1>0$ and $1>s_2,s_3>0$ hold. 
\begin{figure}[htb]{
\begin{overpic}
{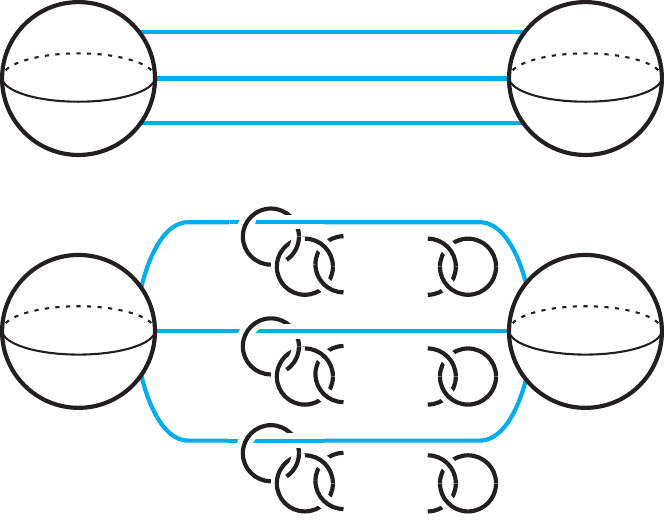}
\put(110, 240){$-1/s_1$}
\put(110, 217){$-1/s_2$}
\put(110, 196){$-1/s_3$}
\put(90, 130){$b^1_{0}$}
\put(110, 118){$b^1_{1}$}
\put(147, 98){$b^1_{2}$}
\put(229, 98){$b^1_{n_1}$}
\put(90, 79){$b^2_{0}$}
\put(110, 65){$b^2_{1}$}
\put(147, 45){$b^2_{2}$}
\put(229, 47){$b^2_{n_2}$}
\put(90, 26){$b^3_{0}$}
\put(110, 12){$b^3_{1}$}
\put(147, -7){$b^3_{2}$}
\put(229, -7){$b^3_{n_3}$}
\end{overpic}}
\caption{Surgery diagrams for $Y(e_0;r_1,r_2,r_3)$ with $e_0\geq 0$ where the $s_i$ are some rational numbers, determined by the $r_i$ as discussed above.}
\label{fig:e0gqe0}
\end{figure}
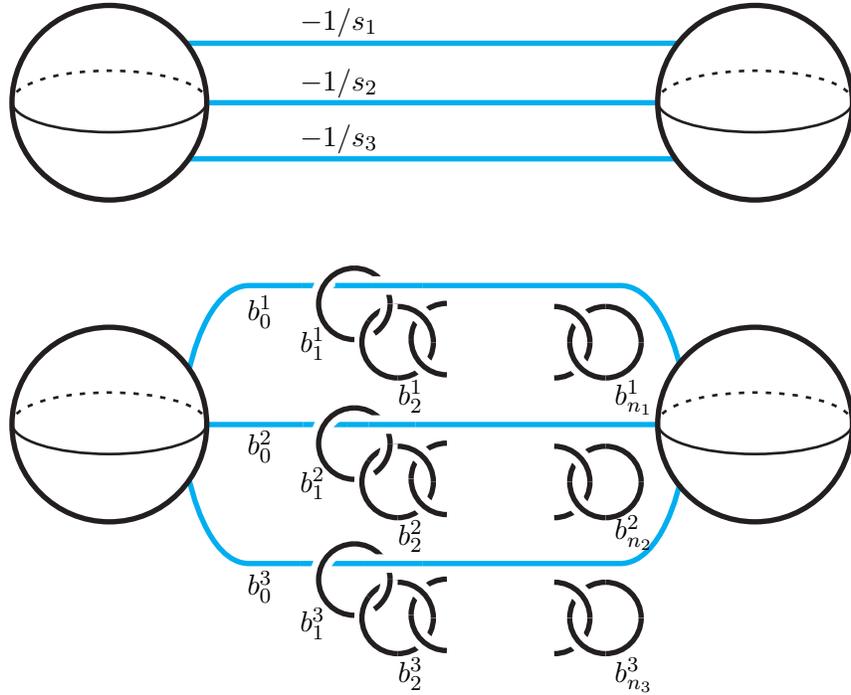
We note that for $s_1\in(n,n+1]$, the surgery diagram will give a small Seifert fibered space with $e_0=n$. 

We may apply the slam dunk move to express the top diagram in Figure~\ref{fig:e0gqe0} in terms of integer surgeries on a chain of unknots attached to each of the blue curves. The resulting configuration is shown in the bottom diagram of the figure. For each $i=1,2,3$, the rational surgery coefficient satisfies $-1/s_i=[b^i_0,\ldots, b^i_{n_i}]$, where all  $b_j^i <-1$ except that $b_0^1< 0$. A contact surgery realizing the top diagram in Figure~\ref{fig:e0gqe0} is equivalent to a Legendrian surgery on the bottom diagram. 

We call one of the blue curves together with the chain of unknots linked to it a \dfn{leg} of the diagram.  A contact structure, as in Definition~\ref{m-consistent}, is called \dfn{consistent} if it is obtained from Legendrian surgery on a Legendrian realization of the bottom surgery diagram in Figure~\ref{fig:e0gqe0}, where all stabilizations on the knots have been performed with the same sign. We say that a contact structure is \dfn{mostly consistent} if it is obtained by Legendrian surgery on a Legendrian realization of the surgery diagram, where all of the stabilizations on each leg are consistent, but not all legs are stabilized the same. We emphasize that a mostly consistent contact structure is not consistent, by definition.

We are now ready to state our results for small Seifert fibered spaces with $e_0\geq 0$. 
\begin{theorem}\label{main3}
Suppose that  $e_0\geq 0$. If  $Y(e_0;r_1,r_2,r_3)$ with a contact structure admits a rational homology ball symplectic filling,  then the contact structure must be mostly consistent. 
\end{theorem}

This theorem and its corollary below will be proven in Section~\ref{pos}.

\begin{remark}
Although Theorem~\ref{main3} provides only a necessary condition, there are small Seifert fibered spaces \(Y(e_0;r_1,r_2,r_3)\) with \(e_0\ge 0\) that admit symplectic rational homology ball fillings. We construct such examples in Section~\ref{s1s2fillings}. See also \cite{EtnyreOzbagciTosun25pre}.   
\end{remark}


\begin{corollary}\label{possfillse0pos}
Suppose that  $e_0\geq 0$. Then, up to isomorphism,  there are at most three contact structures on $Y=Y(e_0;r_1,r_2,r_3)$ so that $(Y, \xi)$ admits a symplectic rational homology ball filling.
\end{corollary}

We now turn to the case of small Seifert fibered spaces with $e_0=-1$. While the tight contact structures on these manifolds are partially classified in \cite{GhigginiLiscaStipsicz2007, Matkovic2018}, much less is known about their symplectic rational homology ball fillings. In what follows, we rule out such fillings for all tight contact structures when the Seifert invariants satisfy certain constraints.  

\begin{theorem}\label{Lspace}
 Let $Y=Y(-1; r_1, r_2, r_3)$, where we assume that $1>r_1\geq r_2\geq r_3>0$,  without loss of generality.
Then there is no tight contact structure $\xi$ on $Y$ so that $(Y, \xi)$ is symplectically filled by a rational homology ball, provided that $r_1+r_2+r_3>1>r_1+r_2$.  
\end{theorem}

We will prove this result in Section~\ref{Lnegm1}.

\smallskip

On the other hand, when $Y(-1; r_1, r_2, r_3)$ does admit a contact structure with a symplectic rational homology ball filling it can admit many.

\begin{theorem}\label{arbitmany}
For any integer $n>0$, there is a small Seifert fibered space $Y=Y(-1; r_1, r_2, r_3)$ which admits at least $n$ distinct tight contact structures, so that $Y$ equipped with any of these contact structures is symplectically filled by a rational homology ball. 
\end{theorem}

This theorem will be established in Section~\ref{SFSonS1S2}.

\subsection{Brieskorn spheres}\label{brieskornS}
A {\em Brieskorn homology sphere} $\Sigma(p,q,r)$\footnote{Here we are only considering Brieskorn homology spheres with three singular fibers as our focus is on small Seifert fibered spaces. We also note that according to a conjecture due to Koll\'{a}r \cite{kollar} and independently Fintushel-Stern it is expected that no Brieskorn homology spheres with more singular fibers can smoothly bound an integral homology ball.}
is the link of the complex surface singularity $$\{z_1^p+z_2^q+z_3^r=0\} \subset \mathbb{C}^3$$ where $p,q,r\geq 2$ are pairwise relatively prime integers. 

It is well known that $\Sigma(p,q,r)$  is a small Seifert fibered space with singular fibers having multiplicities $p,q,r$ \cite{neumannraymond, savelievbook}. When written in normalized Seifert invariants, $\Sigma(p,q,r)$ has $e_0=-1$ or $-2$ \cite[Page~$12$]{MarkTosun22} or \cite[Section~$3.1$]{Tosun22}. 

Throughout this paper, we orient $\Sigma(p,q,r)$ as the link of the singularity $\{z_1^p+z_2^q+z_3^r=0\}$. If $\Sigma(p,q,r)$ is a Seifert fibered space with $e_0=-2$,  then it has $R=1$ \cite[Page~$2$]{NZ85} where $R$ is the homology cobordism invariant for Brieskorn homology spheres introduced by Fintushel and Stern. In particular, by \cite{FS:pseudofree} such a manifold cannot bound a smooth integral homology ball. On the other hand, many $\Sigma(p,q,r)$ (necessarily with $e_0=-1$) are known to bound smooth contractible 4-manifolds \cite{CH, Fickle}. 

We now provide evidence for Gompf’s Conjecture~\ref{con: gompf}  (see \cite{MarkTosun22, Tosun22} for extensive surveys and progress) with the following result.

\begin{theorem}\label{brieskorn}
For any positive  integer $n$ and $(p,q)\neq (2,3)$,  no Brieskorn homology sphere $\Sigma(p,q, pqn+1)$ bounds a symplectic rational homology ball.
\end{theorem}

This theorem and the next one are proven in Section~\ref{surgeryontorus}. The full version of Gompf’s conjecture for Brieskorn homology spheres with their standard orientation was announced in \cite{AlfieriCavallo24pre} using different techniques. 

Note that Theorem~\ref{brieskorn} follows directly from Theorem~\ref{nonLspace},  since $\Sigma(p,q,pqn+1)$ can be  obtained from $S^3$ by $-1/n$-Dehn surgery on the positive $(p,q)$-torus knot \cite[Examples~$1.4$ and $1.5$]{savelievbook}.

\begin{theorem}\label{nonLspace}
Let $p, q$ be relatively prime positive integers and $r$ be any  integer less than $-1$ if $(p,q)=(2,3)$, and otherwise a negative integer or a rational number of the form $-\frac{1}{n}$ for some positive integer $n$. 
Then $S^3_{T_{p,q}}(r)$, the result of $r$-Dehn surgery on the positive torus knot $T_{p,q}$,  does not bound a symplectic rational homology ball.
\end{theorem}

\begin{remark}
As we will see during the proof of Theorem~\ref{nonLspace}, our arguments cannot rule out a symplectic rational homology ball filling of the Brieskorn sphere $\Sigma(2,3,6n+1)$. We note that for any non-negative integer $k$,  the Brieskorn sphere $\Sigma(2, 3, 6(2k+1)+1)$ does not bound a smooth integer homology ball since its Rokhlin invariant is not zero. On the other hand, the complete answer for the infinite family  $\Sigma(2,3,6(2k)+1)$  is not known, although some members (e.g. $\Sigma(2,3,13)$ and $\Sigma(2,3,25)$) of this infinite family do bound smooth contractible 4-manifolds. See \cite[Theorem~1.7]{MarkTosunT18} and \cite{Tosun20} for some partial negative results.
\end{remark}

The authors have not been able to find any $Y(-1;r_1,r_2,r_3)$ (or $Y(-2;r_1,r_2,r_3)$) that are not $L$-spaces but yet admit a symplectic rational homology ball filling. So, we end this section by asking the following curious question,  for which a negative answer will not just resolve Gompf's conjecture in full generality but also provide a deeper explanation for it. 

\begin{question}
Is there a small Seifert fibered space that is not an $L$-space but which has a symplectic rational homology ball filling? 
\end{question}

\subsection{Spherical $3$-manifolds}
A closed, orientable 3-manifold is called spherical if it admits a complete
metric of constant curvature $+1$. Equivalently, a spherical $3$-manifold is  the quotient manifold of the form $S^3/G$, where $G$ is a finite subgroup of $SO(4)$ acting freely by rotations on $S^3$. Note that the fundamental group of the spherical $3$-manifold $S^3/G$ is isomorphic to $G$ and conversely, by Perelman's elliptization theorem, any closed orientable prime $3$-manifold with finite fundamental group is spherical. The family of spherical $3$-manifolds are categorized into five types with respect to their
fundamental groups: {\bf C} (cyclic), {\bf D} (dihedral), {\bf T} (tetrahedral), {\bf O} (octahedral) and {\bf I} (icosahedral).   The lens spaces are precisely the class of spherical $3$-manifolds of type  {\bf C}. 

The family of spherical $3$-manifolds can be identified with the homeomorphism types of the links of quotient surface singularities. The  lens spaces are precisely the links of cyclic quotient surface singularities.  It follows that every spherical $3$-manifold, viewed as the {\em oriented} link of a normal surface singularity, has a canonical (a.k.a Milnor fillable) contact structure $\xi_{can}$, as we discussed in Section~\ref{subsec: fillsfs}. Recall that $\mathcal{O}$ in the theorem below is defined by Equation~(\ref{eq: O}).

\begin{theorem}\label{thm: spherical} Suppose that  $\xi$ is a contact structure on a spherical $3$-manifold $Y$, oriented as the link of the corresponding quotient surface singularity. If $(Y, \xi)$ admits a symplectic rational homology ball  filling,  then $Y$ is orientation-preserving diffeomorphic to a lens space $L(p,q)$ with $p/q \in \mathcal{O}$, and $\xi$ is contactomorphic  to $\xi_{can}$. 
\end{theorem}

We will prove this theorem in Section~\ref{sphericalsec}. 

\begin{remark} If the surface singularity is cyclic quotient, then its oriented link $Y$ is a lens space and the result in Theorem~\ref{thm: spherical} for this case was already proven \cite{ChristianLi23, EtnyreRoy21, EtnyreTosun2023, GollaStarkston2022, Lisca2007} as we mentioned above.  Here we give yet another proof of this fact using a different line of argument. 
\end{remark}
It is interesting to note that there are spherical $3$-manifolds, other than some lens spaces, which smoothly bound rational homology balls.
\begin{theorem}[Choe and Park 2021, \cite{ChoePark2021}] \label{thm: ball}  A  spherical $3$-manifold $Y$ bounds a smooth rational homology ball if and only if $Y$ or $-Y$ is homeomorphic to one of the following manifolds: 
\begin{enumerate}
\item $L(p, q)$ such that $p/q \in\mathcal{R}$,
\item $D(p, q)$ such that $(p-q)/q' \in\mathcal{R}$,
\item $T_3$, $T_{27}$ and $I_{49}$, 
\end{enumerate}
where $p$ and $q$ are relatively prime integers such that $0 < q < p$, and $0 < q' < p - q$ is
the reduction of $q$ modulo $p - q$. 
\end{theorem}

We refer to Section~\ref{sphericalsec}, for the description of the spherical $3$-manifolds $D(p,q)$, $T_3$, $T_{27}$ and $I_{49}$ as plumbing graphs.

\begin{remark}\label{rem: notTrue}
The conclusion of Theorem~\ref{thm: spherical} does not necessarily remain valid for a spherical 
$3$--manifold when it is equipped with the orientation opposite to the canonical one 
that arises when the manifold is viewed as a singularity link. For instance,
\[
-T_{3} \;=\; Y(-1;\tfrac{2}{3},\,\tfrac{1}{2},\,\tfrac{1}{3})
\]
admits two non--isotopic tight contact structures, both of which bound symplectic 
rational homology balls, as illustrated in Figure~\ref{fig:mt3}. By contrast, neither
\[
-T_{27} \;=\; Y(3;\tfrac{2}{3},\,\tfrac{1}{2},\,\tfrac{1}{3}) 
\quad \text{nor} \quad 
-I_{49} \;=\; Y(0;\tfrac{4}{5},\,\tfrac{1}{2},\,\tfrac{1}{3})
\]
admit symplectic rational homology ball fillings. We establish these results as part 
of a more general framework in the sequel \cite{EtnyreOzbagciTosun25pre}.
\end{remark} 

\begin{figure}[htb]{
\begin{overpic}
{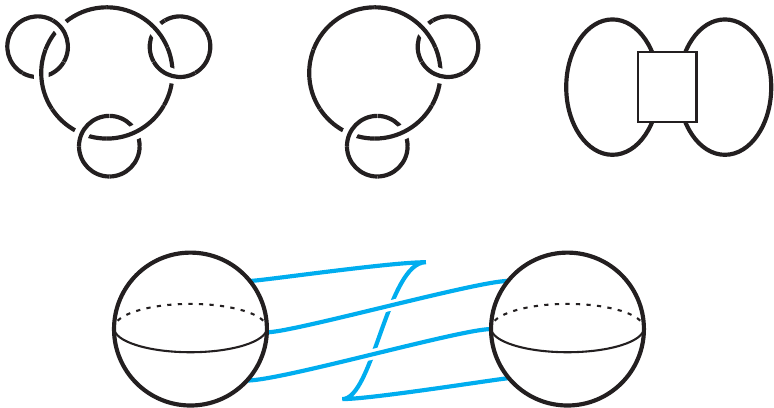}
\put(-24, 180){$-3/2$}
\put(105, 180){$-3$}
\put(68, 107){$-2$}
\put(11, 130){$-1$}
\put(232, 180){$-3$}
\put(195, 105){$-2$}
\put(127, 130){$-1/3$}
\put(319, 150){$3$}
\put(268, 125){$0$}
\put(372, 130){$1$}
\end{overpic}}
\caption{Top left: $-T_{3} = Y(-1;\tfrac{2}{3},\tfrac{1}{2},\tfrac{1}{3})$. 
Moving to the right, we perform a slam dunk followed by three Rolfsen twists. 
In the bottom diagram, the blue curve represents a Legendrian knot on the 
boundary of $S^{1} \times D^{3}$ with $tb=-3$. With one additional stabilization, 
for which there are two natural choices, this diagram realizes Stein rational 
homology ball fillings corresponding to two non--isotopic contact structures on 
the boundary $-T_{3}$. 
(Note that in passing from the top right diagram to the bottom diagram, 
we replace the $0$--framed unknot with a $1$--handle and push the overstrand 
to an understrand, a standard Kirby calculus move that changes the smooth 
framing by $-6$.)}
\label{fig:mt3}
\end{figure}

\subsection{Lagrangian fillings of Legendrian knots in Seifert fibered spaces}

Our discussion  above for small Seifert fibered spaces that bound symplectic rational homology balls has an interesting consequence. Recall that it has been of interest \cite{HaydenSabloff2015}  to determine when a Legendrian knot in a contact 3–manifold bounds a Lagrangian disk (or more generally, a Lagrangian surface) in a symplectic filling. We emphasize that whether or not this can occur depends strongly on the choice of symplectic filling.

\begin{theorem}\label{lagslice}
Let $Y=Y(e_0;r_1,r_2,r_3)$ be a small Seifert fibered space with $e_0\leq -3$ such that  $(Y, \xi_{can})$ admits a symplectic rational homology ball filling. We know that $(Y, \xi_{can})$ also has a symplectic filling $(X,\omega)$ obtained by adding Weinstein $2$-handles to $B^4$ along a Legendrian realization of the diagram in Figure~\ref{fig:plumbing} where all the Legendrian unknots have been stabilized consistently (that is, all on the left, or all on the right). Let $L$ be a belt sphere for one of the $2$-handles in $X$. By construction, $L$ bounds a Lagrangian disk in $(X,\omega)$. However, the Legendrian knot $L$ cannot bound a Lagrangian disk in the symplectic rational homology ball filling. 
\end{theorem}
We prove this result at the end of Section~\ref{Mneg}.
\begin{remark}
When $e_0>-3$, we construct many examples of contact structures that admit symplectic rational homology ball fillings. They can also be obtained by attaching Weinstein $2$-handles to $B^4$ along a Legendrian realization of the link in Figure~\ref{fig:plumbing} or to $S^1\times D^3$ along a Legendrian realization of the link in Figure~\ref{fig:e0gqe0}. The belt spheres to these $2$-handles are Legendrian knots that bound Lagrangian disks in the $2$-handles. However, as in the proof of Theorem~\ref{lagslice}, these Legendrian knots do not bound Lagrangian disks in their corresponding symplectic rational homology ball fillings. 
\end{remark}

\subsection{Outline} 
Most of our theorems follow from two main techniques. The first is a result of Christian and Menke which provides a method  to split symplectic fillings of a contact manifold under certain conditions. The second involves computations and estimates on the $\theta$-invariant of a contact structure. 

In Section~\ref{bg}, we recall the necessary contact geometry background for Christian and Menke's result before stating it. We also review the definition of the $\theta$-invariant together with some of its basic properties. Section~\ref{sec: compare}  is devoted to establishing the estimate on the  $\theta$-invariant given in Proposition~\ref{canminimizes}.

In Section~\ref{SFSonS1S2}, we discuss Seifert fibered structures on $S^1\times S^2$, construct symplectic rational homology ball fillings of some Seifert fibered spaces, and prove  Theorem~\ref{arbitmany}. We then proceed to the proofs of our main theorems: Section~\ref{Mneg} covers the case $e_0 \leq -3$ (proving Theorem~\ref{main1}), Section~\ref{Lnegm2}  treats the case $e_0 \leq -2$ (proving Theorem~\ref{main2} and Corollary~\ref{possiblefillingsfore0m2}), and Section~\ref{pos} handles the case $e_0 \geq 0$ (proving Theorem~\ref{main3}  and Corollary~\ref{possfillse0pos}).

Section~\ref{Lnegm1} is devoted to Seifert fibered spaces with $e_0=-1$, where we prove Theorem~\ref{Lspace}, as well as to Brieskorn spheres, where we establish  Theorem~\ref{brieskorn} and Theorem~\ref{nonLspace}. Finally, in Section~\ref{sphericalsec}, we prove Theorem~\ref{thm: spherical}, which characterizes canonically oriented spherical $3$-manifolds that admit symplectic rational homology ball fillings.

\vspace{.2in}
\noindent
{\bf Acknowledgments:} We thank Jon Simone for helpful communications. We also thank the anonymous referee for providing valuable comments that improved the paper. The first author was partially supported by National Science Foundation grant DMS-2203312 and the Georgia Institute of Technology’s Elaine M. Hubbard Distinguished Faculty Award. The third author was supported in part by grants from National Science Foundation (DMS-2105525 and CAREER DMS 2144363) and the Simons Foundation (636841, BT and 2023 Simons Fellowship). He also acknowledges the support by the Charles Simonyi Endowment at the Institute for Advanced Study.

\section{Background and preliminary results}\label{bg}

In this section, we collect various preliminary results. We begin by discussing the classification of contact structures on solid tori and lens spaces in Section~\ref{ctstronsolidtori}. In the following two sections we recall various facts about smooth and contact surgeries respectively. In Section~\ref{ctstronsolidtori} we discuss a theorem of Christian and Menke that shows that a symplectic filling of a contact manifold can sometimes be ``decomposed" if there is a ``mixed" torus in the contact manifold. In the final subsection we recall the definition of the $\theta$-invariant and its relation to symplectic rational homology balls. 

\subsection{Contact structures on solid tori and lens spaces}\label{ctstronsolidtori}
We recall the classification of tight contact structure on a solid torus. For this we need to use the Farey graph to describe curves on a surface and the notion of convex surfaces. We refer to \cite{EtnyreRoy21} for these well-known concepts, as the conventions we use here are used there (except when drawing a portion of the Farey graph as in Figure~\ref{fig:cfb} moving left to right corresponds to moving clockwise in the Farey graph, where in \cite{EtnyreRoy21} it was anti-clockwise). 

We begin by describing notation for paths in the Farey graph. Let $s$ and $t$ be two vertices in the Farey graph that have an edge between them. To be specific, we take $t$ to be anti-clockwise of $s$. The sequence of points $v_0=s$, $v_1=s\oplus t$, $v_2=v_1\oplus t=s\oplus 2t, \ldots, v_k=s\oplus kt$ is called a \dfn{continued fraction block}. (Here $s\oplus t$ is the Farey sum of the two rational numbers, which just means that the numerator of the new fraction is the sum of the numerators of the summands and similarly for the denominator.) We call $s$ the \dfn{start} of the continued fraction block and $t$ the \dfn{target}. See Figure~\ref{fig:cfb}. 
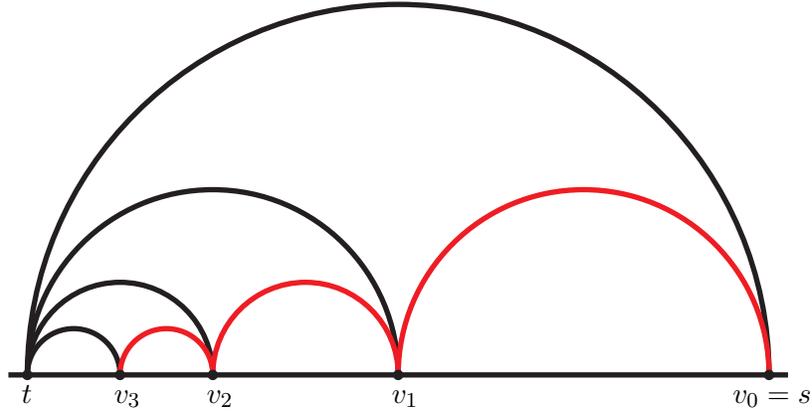
\begin{figure}[htb]{
\begin{overpic}
{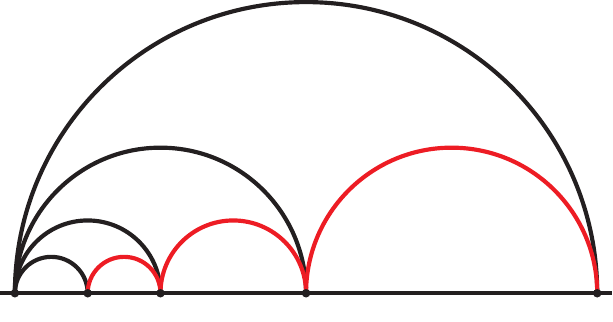}
\put(273, 3){$v_0=s$}
\put(144, 3){$v_1$}
\put(72, 3){$v_2$}
\put(39, 3){$v_3$}
\put(3, 3){$t$}
\end{overpic}}
\caption{A continued fraction block. The Farey graph is usually drawn on the unit disk in $\R^2$, but for convenience, we put the vertices of the graph on $\R$ and moving left to right on $\R$ is going clockwise on the boundary of the unit disk. }
\label{fig:cfb}
\end{figure}
We note that each pair of consecutive points in the continued fraction block are connected by an edge and the union of all the edges between adjacent points forms a shortest path from $s$ to $s\oplus kt$ in the Farey graph. To describe more general shortest paths in the Farey graph, we need to describe notation for adjacent continued fraction blocks. Given the continued fraction block above, let $s'=v_k=s\oplus kt$, $t_1=s\oplus (k+1)t$, and $t_l=t_1\oplus (l-1) s'$. See Figure~\ref{fig:ndown}.
\begin{figure}[htb]{
\begin{overpic}
{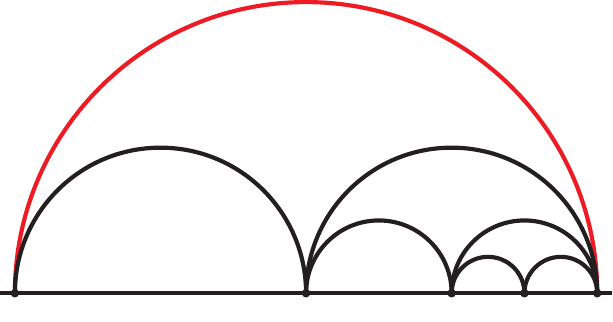}
\put(-48, 2){$t_1=s\oplus (k+1)t$}
\put(145, 2){$t_2$}
\put(215, 2){$t_3$}
\put(248, 2){$t_4$}
\put(273, 2){$s'=s\oplus kt$}
\end{overpic}}
\caption{The vertices $t_i$ are the potential target vertices for the next continued fraction block and the next jumps in the minimal path with start $s'$ and target $t_l$ will be of the form $\{s'\oplus  t_l, \ldots,  s'\oplus m t_l, \}$ for some $l$ and $m$.}
\label{fig:ndown}
\end{figure}
The continued fraction block above from $s$ to $s'$ can be continued as a shortest path by adding a continued fraction block with start $s'$ and target $t_l$ for some $l$. We say that this next continued fraction block is $l$ down from the first continued fraction block. Any shortest path in the Farey graph is composed of a continued fraction block followed by a continued fraction block that is $l$ down from it, followed by another continued fraction block $l'$ down from the last, and so on.
\begin{remark}\label{possE}
Below we will need an observation about paths in the Farey graph.
Given the continued fraction block above, the set of points clockwise of $v_{i-1}$, anti-clockwise of $v_{i+1}$, and having an edge to $v_i$ is $\{t\}$. Similarly the set of points clockwise of $t_{1}$, anti-clockwise of $t_l$ with an edge to $s'=v_k$ is $\{t_2, t_3, \ldots, t_{l-1}\}$. 
\end{remark}

If we fix a basis $\{\lambda,\mu\}$ for $\Z^2\cong H_1(T^2)$, then any embedded curve $\gamma$ on $T^2$ is determined by its homology class, which can be described by a pair of relatively prime integers $(a,b)$: $[\gamma]=a\lambda + b\mu$. Giving a pair of relatively prime integers $(a,b)$ is equivalent to giving a fraction $b/a$ which is an element of $\Q\cup \{\infty\}$. 

If $s_0$ and $s_1$ are slopes connected by an edge in the Farey graph, then there will be exactly two tight (minimally twisting) contact structures on $T^2\times [0,1]$ with convex boundary having two dividing curves of slope $s_0$ on $T^2\times\{0\}$ and $s_1$ on $T^2\times \{1\}$. These are called \dfn{basic slices} and we call one of them a positive basic slice and the other a negative basic slice. Giroux \cite{Giroux00} and Honda \cite{Honda00a} showed that any tight contact structure on $T^2\times [0,1]$ can be decomposed into the union of basic slices and this union will give a path in the Farey graph with each edge decorated by a $+$ or $-$ sign depending on the type of basic slice it corresponds to. Specifically, we have the following classification result for a minimally twisting contact structure, but first we recall that a contact structure on $T^2\times[0,1]$ with convex boundary having dividing slope $s_i$ on $T^2\times \{i\}$, for $i=0,1$, is minimally twisting if any convex torus in $T^2\times [0,1]$ that is smoothly isotopic to the boundary has dividing slope clockwise of $s_0$ and anticlockwise of $s_1$ in the Farey graph.
\begin{theorem}[Giroux 2000, \cite{Giroux00} and Honda 2000, \cite{Honda00a}]\label{thickenedT2}
Minimally twisting tight contact structures on $T^2\times[0,1]$ having convex boundary with two dividing curves of slope $s_i$ on $T^2\times \{i\}$ for $i=0,1$, are in one-to-one correspondence with minimal decorated paths in the Farey graph from $s_0$ clockwise to $s_1$, up to shuffling signs in continued fraction blocks. 
\end{theorem}
We note that a corollary of this theorem, discussed in \cite{Honda00a}, says that if we have a non-minimal path in the Farey graph between $s_0$ and $s_1$ with decorations on it, we can use it to build a contact structure by stacking basic slices according to the decorations. If the path can be shortened by removing two adjacent edges with opposite signs, then the contact structure is overtwisted. 

We describe solid tori as follows. Consider $T^2\times [0,1]$ and 
let $S_s$ be the quotient of $T^2\times[0,1]$ after collapsing the  leaves of a linear foliation of slope $s$ on $T^2\times\{0\}$. We will call $S_s$ the \dfn{solid torus with (lower) meridian of slope $s$.} We say a minimal path in the Farey graph from $s$ clockwise to $r$ is \dfn{partially decorated} if all the edges of the path have been assigned a $+$ or $-$ sign except the first edge, which is left blank. 
\begin{theorem}[Giroux 2000, \cite{Giroux00} and Honda 2000, \cite{Honda00a}]
Tight contact structures on a solid torus with meridional slope $s$ and convex boundary with two dividing curves of slope $r$ are in one-to-one correspondence with partially decorated minimal paths in the Farey graph from $s$ clockwise to $r$, up to shuffling signs in continued fraction blocks. 
\end{theorem}
It will sometimes be helpful to consider the solid torus in a different way. In the construction above, we could let $S^s$ be the quotient of $T^2\times [0,1]$ after collapsing the leaves of a linear foliation of slope $s$ on $T^2\times \{1\}$. We say that $S^s$ is the \dfn{solid torus with upper meridian $s$}. Notice that the natural orientation on $\partial S^s$ coming from $T^2\times [0,1]$ is inward pointing, while on $\partial S_s$ it is outward pointing. The classification theorem above is the same for $S^s$ except in the definition of partial decoration, it is the last edge that is left blank. (It is always the edge adjacent to the meridian that is left blank.) 

In describing contact surgery below, we will mainly be interested in the situation where the dividing slope is $0$. For the classification of contact structures in the situation, we consider 
a negative surgery coefficient given by the continued fraction $[a_0,\ldots, a_n]$ where
$a_i<-1$ for all $i>0$ and $a_0 < 0$.

We can use this continued fraction to describe a minimal path in the Farey graph from $0$ anti-clockwise to $[a_0,\ldots, a_n]$. We start by assuming that none of the $a_i$ are $-2$ except possibly for the first or the last. Begin with a continued fraction block having start $s_0=0$ and target $t_0=\infty$. We then consider $v_i=s_0 \oplus it_0$ for $i=0$ to $|a_0+1|$. So $v_{|a_0+1|}=a_0+1$. The next continued fraction block will have start $s_1= s_0\oplus |a_0+1|t_0$, target $t_1=s_0\oplus |a_0|t_0$, and make $|a_1+2|$ jumps. In general, if we have described the path to the end of the $k^{th}$ continued fraction block, which had start $s_k=s_{k-1}\oplus |a_{k-1}+2|t_{k-1}|$,  target $t_k=s_{k-1}\oplus |a_k+1|t_{k-1}$, and vertices $s_k, s_k\oplus t_k,\ldots, s_k\oplus |a_k+2|t_k$, we start the next continued fraction block with start $s_{k+1}=s_k\oplus |a_k+2|t_k$, target $t_{k+1}=s_k\oplus |a_k+1|t_k$, and the path will have $|a_k+2|$ jumps, except the last continued fraction block will make $|a_n+1|$ jumps. This will give a minimal path from $0$ to $[a_0,\ldots, a_n]$. 
\begin{example}
Consider $[-3, -3, -2]$. According to the above discussion, $a_0=-3$ tells us to consider the vertices $0, -1, -2$ (that is, making $2=|-3+1|$ jumps along a continued fraction block). Then $a_1=-3$ tells us to make $1=|-3+2|$ jump along a continued fraction block with starting vertex $-2$ and target vertex $-3$. That is, we make a jump to $-5/2$. Finally, for $a_2=-2$ we make $1=|-2+1|$ jump in the continued fraction block with start $-5/2$ and target $-8/3$. Thus, the last jump will be to $-13/5$. One may easily check that $-13/5=[-3,-3,-2]$. 
\end{example}
\noindent
If $a_k$ and $a_{k+l+1}$ are not $-2$, or are the first or last entries in the continued fraction, but the $a_i$ in between them are $-2$, then after the continued fraction block associated to $a_k$, we have the continued fraction block associated to $a_{k+l+1}$ that is $l$ down from the previous block. 
\begin{example}
We now consider $[-2,-2,-2]$. The first $-2$ gives us the path from $0$ to $-1$ (that is $|-2+1|$ jump from $0$ towards $\infty$). The next $-2$ does not contribute to a jump, but tells us that we need to move down to a lower continued fraction block. The last $-2$ tells us to make $1$ jump in a continued fraction block that is $2$ down from the previous one. That is, the starting vertex will be $-1$, and the target vertex will be $-3/2$. So the last jump will be to $-4/3$. One may easily check that $-4/3=[-2,-2,-2]$. 
\end{example}

Notice that the above construction produces a path with $n+1$ continued fraction blocks with blocks having $|a_0+1|, |a_1+2|, \ldots, |a_{n-1}+2|, |a_n+1|$ edges, respectively. From this the theorem below easily yields.
\begin{theorem}\label{surgerytori}
Let $r<0$ be a rational number with continued fraction $[a_0,\ldots, a_n]$. Then on the solid torus with dividing slope $0$ and meridional slope $r$ there are
\[
|a_0(a_1+1)\cdots(a_n+1)|
\]
tight contact structures. 
\end{theorem}

Moving to lens spaces we recall that $L(p,q)$ is the result of $-p/q$ surgery on the unknot. We can describe this in another way. Consider $T^2\times [0,1]$ and take the quotient where on $T^2\times\{1\}$ collapse the leaves of a linear foliation of slope $0$ and on $T^2\times \{0\}$ collapse the leaves of a linear foliation of slope $-p/q$. This will also result in $L(p,q)$. So we can think of $L(p,q)$ as the result of the union of the solid tori $S^0$  and $S_{-p/q}$. Work of Giroux \cite{Giroux00} and Honda \cite{Honda00a} show  that any tight contact structure on $L(p,q)$ can be described by a path in the Farey graph from $-p/q$ clockwise to $0$ with decorations on all but the first and last edges. 

\subsection{Smooth surgery}
Here we recall how to construct various Dehn surgery diagrams for the same smooth $3$-manifold. It is well-known \cite{Rolfsen76} that the slam dunk operation relates the two surgery diagrams in Figure~\ref{fig:sd},  where $r=[a_0,\ldots, a_n]$. 
\begin{figure}[htb]{
\begin{overpic}
{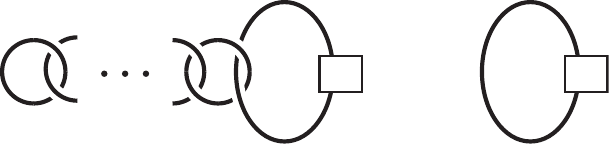}
\put(158, 5){$a_0$}
\put(100, 10){$a_1$}
\put(7, 10){$a_n$}
\put(278, 10){$r$}
\put(160, 32){$K$}
\put(278, 32){$K$}
\end{overpic}}
\caption{Slam dunk operation.}
\label{fig:sd}
\end{figure}
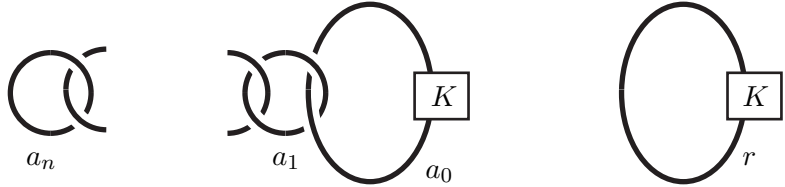
When discussing contact surgery below it will also be useful to have the ``rolled up" version of the left-hand side of Figure~\ref{fig:sd}. Starting with that diagram slide the $a_1$-framed unknot over the knot $K$ to get $K_1$ and then the $a_2$-framed unknot over $K_1$, and continuing in this manner until the $a_n$-framed unknot is slid. This will result in Figure~\ref{fig:rolled} which also realizes $r$ surgery on $K$, where $b_i=2i+\sum_{j=0}^i a_j$ for $ i \geq 0$, $c_1=a_1+1$ and $c_i=a_i+2$ for $i>1$. 
\begin{figure}[htb]{
\begin{overpic}
{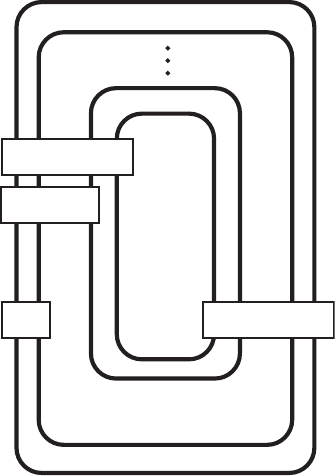}
\put(100, 172){$b_0$}
\put(110, 188){$b_1$}
\put(128, 216){$b_{n-1}$}
\put(152, 224){$b_n$}
\put(120, 71){$K$}
\put(8, 72){$c_{n}$}
\put(20, 127){$c_2$}
\put(30, 151){$c_1$}
\end{overpic}}
\caption{Rolled up diagram.}
\label{fig:rolled}
\end{figure}
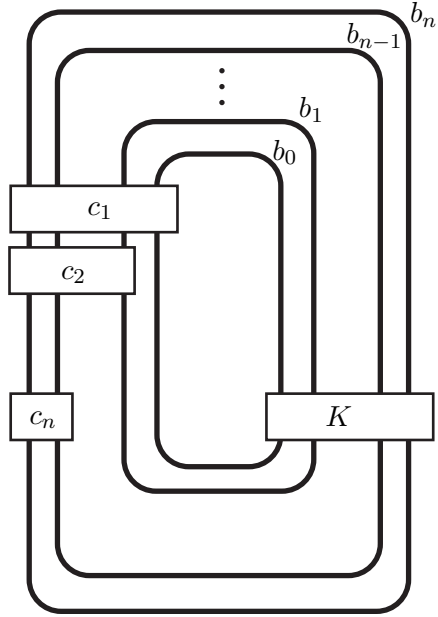

\subsection{Contact surgery}
Give a Legendrian knot $L$ in a contact $3$-manifold $(M,\xi)$, \dfn{contact $(r)$-surgery} on $L$ is the result of removing a standard neighborhood $N$ of $L$ from $M$ and gluing in a new solid torus with meridional slope $r+\tb(L)$ and then extending the contact structure over this torus to be some tight contact structure on the solid torus.  Notice that the resulting $3$-manifold is obtained from $M$ by smooth $r+\tb(L)$ surgery on $L$ and there are many possibilities for the contact structure on the solid torus that was glued into $M\setminus N$. So for a general $r$ there are many different possible contact structures that result from contact $(r)$-surgery. To see how many notice that if we used the contact framing to describe the Dehn surgery coefficient then it would be $r$. Thus the number of possible contact structures on the surgered manifold is given by Theorem~\ref{surgerytori}. As discussed in Section~\ref{ctstronsolidtori},  the contact structure will be determined by a partially decorated minimal path in the Farey graph from $r+\tb(L)$ clockwise to $\tb(L)$. 

Given a Legendrian knot $L$, suppose we perform contact $(r)$-surgery on $L$ with $r<0$. If $r=[a_0,\ldots, a_n]$ then Theorem~\ref{surgerytori} gives the number of possible contact structures obtained from such a surgery. We interpret this in terms of Legendrian surgery. Specifically, let $L_0$ be $L$ stabilized $|a_0+1|$ times, let $L_1$ be a Legendrian push-off of $L_0$ stabilized $|a_1+2|$ times, and continuing so that $L_k$ is the Legendrian push-off of $L_{k-1}$ stabilized $|a_k+2|$ times (recall that all the $a_i$ are less than $-1$ for $i>0$). Legendrian surgery on the link $L_0\cup \cdots \cup L_n$ will be equivalent to smooth $r+\tb(L)$ surgery on $L$. We also note that the number of different stabilizations at the $k^{th}$ stage of this construction corresponds to the $k^{th}$ continued fraction block in the path in the Farey graph from $0$ anti-clockwise to $r$ and hence the different choices of stabilizations give all the contact structures on the contact $(r)$-surgery. One may also check that the Legendrian realizations of the diagram on the left in Figure~\ref{fig:sd} also give all possible contact structures coming from contact $(r)$-surgery and again the different choices of stabilizations correspond to the different choices of sign in the path in the Farey graph.

\subsection{Decomposing symplectic fillings}\label{CM}
Here we discuss recent work of Christian and Menke that allows us to decompose certain symplectic fillings of contact manifolds based on special tori in the contact manifold. 
\begin{definition} A torus $T$ in a contact $3$-manifold $(M,\xi)$ is called a \dfn{mixed torus} if it is convex and contained in a neighborhood $N=T^2\times [0,1]$ such that $T$ splits $N$ into two basic slices with opposite sign. Let $s$ be the slope of the dividing curves on $T$ and $s_i$ the slope of the dividing curves on $T^2\times \{i\}$ for $i=0,1$. Let $E_T$ be the set of slopes in the Farey graph that are clockwise of $s_1$ and anti-clockwise of $s_0$ with an edge to $s$. We call these \dfn{exceptional slopes}. For any $e\in E_T$ we say the \dfn{$e$-splitting of $(M,\xi)$} is the result of cutting $M$ along $T$ and gluing in two solid tori with meridional slope $e$ to the resulting boundary components and extending $\xi|_{M\setminus T}$ over these solid tori with dividing slope $s$. 
\end{definition} 

Notice that since there is an edge in the Farey graph between $s$ and $e$, there is a unique tight contact structure on a solid torus with meridional slope $e$ and dividing slope $s$. Also note that inside the $e$-splitting, the cores of the added solid tori are Legendrian knots with standard neighborhoods given by the added tori. 

\begin{theorem}[Christian-Menke 2018, \cite{ChristianMenke25pre}]\label{splittingthm}
If $(X,\omega)$ is a (weak) symplectic filling of the contact manifold $(M,\xi)$ and $T$ is a mixed torus in $(M,\xi)$, then there is a symplectic manifold $(X',\omega')$ that is a weak symplectic filling of the result of $e$-splitting $(M,\xi)$ along some slope $e\in E_T$. Moreover, $(X,\omega)$ may be recovered from $(X',\omega')$ by attaching a round Weinstein $1$-handle to the Legendrian knots defined by the $e$-splitting. 
\end{theorem}

We recall how to attach a round Weinstein $1$-handle. Given two Legendrian knots $L_1$ and $L_2$ in the boundary of a symplectic manifold $W$ with convex boundary. We may attach a Weinstein $1$-handle to $W$ so that the attaching sphere consists of a point on $L_1$ and a point on $L_2$. This may be done so that in the boundary of the new symplectic manifold, we see the connected sum of $L_1$ and $L_2$. We now attach a Weinstein $2$-handle to $L_1\# L_2$. Attaching both these handles is what we mean by attaching a round Weinstein $1$-handle to $W$. 

\begin{remark}
We note that the statement of \cite[Theorem 1.1]{ChristianMenke25pre} looks  different than the one we give in Theorem~\ref{splittingthm}. Since they did not discuss the Farey graph, they stated a version of the theorem for specific values of $s_0,s_1$, and $s$. Nevertheless, by a change of coordinates on the torus, Theorem~\ref{splittingthm} above exactly agrees with \cite[Theorem 1.1]{ChristianMenke25pre}.
\end{remark}

\begin{remark}
We note that all of the symplectic fillings discussed in the introduction were implicitly strong symplectic fillings, but since any symplectic filling of a contact structure on a rational homology ball can be deformed to a strong filling \cite{Eliashberg91, OhtaOno99}, we can ignore the distinction for the contact manifolds we are considering. However, in our arguments below, $(X',\omega)$ might only be a weak symplectic filling of its boundary; this will not be a problem, since are arguments rulling out rational homology ball symplectic fillings of the original $(M,\xi)$ will only rely on the topology of $(X',\omega)$ and not its geometry. 
\end{remark}

\begin{remark}\label{disconnected}
We note that if $T$ is a separating torus in the above theorem, then $\partial X'$ is disconnected. Moreover, it is known that if a contact manifold is supported by a planar open book, then any symplectic filling must have a connected boundary \cite{Etnyre04b}. Thus, if one of the components of $\partial X'$ is supported by a planar open book, then $X'$ is disconnected and each component of the boundary is filled by one of the components of $X'$. Below, we will be interested in the case where one of the components of the $e$-splitting is a lens space. These are all known to be supported by planar open books \cite{Schoenenberger05}, so in these cases, $X'$ is disconnected. 
\end{remark}

\begin{remark}\label{homologyofsplitting}
We make a simple observation about the homology of $X$ and its relation to $X'$. We will be interested in cases when $X'=X_1\cup X_2$. In this case, a simple homology computation shows that the Betti numbers satisfy 
\[
b_1(X)=b_1(X_1)+b_1(X_2)-1 \text{ and } b_i(X)=b_i(X_1)+b_i(X_2)
\] for $i=2,3$, or 
\[
b_2(X)=b_2(X_1)+b_2(X_2)+1 \text{ and } b_i(X)=b_i(X_1)+b_i(X_2)
\]
for $i=1,3$. To see this, recall that a round $1$-handle decomposes uniquely as a union of a $1$-handle and a $2$-handle passing over the $1$-handle geometrically twice and algebraically zero times. If we first attach a $1$-handle to connect $X_1$ and $X_2$, then the Betti number $b_i$ of the resulting $4$-manifold will be given by $b_i(X_1) + b_i(X_2)$, for $i=1,2,3$.    When we attach the $2$-handle as described above, then either this $2$-handle cancels against a $1$-handle in one of the pieces or it does not. If it does, then we have the first case, and otherwise we have the second case above.

Thus, if $X$ is a rational homology ball, then one of $X_i$ is a rational homology ball and the other is a rational homology $S^1$.
\end{remark}

\begin{remark}\label{splits}
Suppose $(M,\xi)$ is obtained from $(M',\xi')$ by contact $(r)$-surgery on $L$ where the contact surgery is described by a partially decorated path $P$ in the Farey graph. If the path has a continued fraction block containing both signs, then there is a mixed torus $T$ and $E_T=\{t\}$ where $t$ was the target vertex for the continued fraction block. See Remark~\ref{possE}. Thus cutting along $T$ will produce a solid torus and $M'$ with a neighborhood of $L$ removed and we see an $e$-splitting will give a lens space and the result of a contact $(r')$-surgery on $L$ where the continued fraction expansion of $r'$ has a shorter length than that of $r$. 

An important special case of this is when the first continued fraction block has mixed signs. In that case $E_T=\{\infty\}$ so the $e$-splitting of $M$ will give a lens space and $M'$. Thus, the filling of $M$ will come from a filling of $M'$ and a filling of the lens space by attaching a round $1$-handle. 

Similarly, if two adjacent continued fraction blocks had different signs, and $T$ was the torus separating the continued fraction blocks, then $T$ is a mixed torus and 
\[
E_T=\{t_1, t_2, \ldots, t_{l-1}, s\oplus (k+1) t\}.
\] 
As above we see that for any $e\in E_T$, the $e$-splitting of $M$ will consist of a lens space and the result of some contact $(r')$-surgery on $L$.
\end{remark}

\subsection{An invariant of plane fields}\label{recalltheta}
Let $\xi$ be an oriented tangent $2$-plane field on an oriented rational homology $3$-sphere $Y$.
Gompf \cite{Gompf98} showed that the homotopy class of $\xi$
 is determined by its induced spin$^c$ structure  and its $3$-dimensional invariant $\theta(\xi) \in \mathbb{Q}$, which is defined as follows. For any given pair $(Y, \xi)$ as above, there is  a compact smooth $4$-manifold $X$, equipped with an almost-complex structure $J$, so that $\partial X = Y$, and $\xi$ is homotopic to the oriented $2$-plane field $TY \cap JTY$ (the complex tangencies of $Y$).  Let $\chi(X)$ and $\sigma(X)$ denote the Euler characteristic and the signature of $X$, respectively and let $c_1(X, J)$ denote the first Chern class of the almost complex manifold $(X, J)$. Then it turns out that the rational number  $c^2_1 (X, J)-2\chi(X)-3\sigma(X)$ is independent of the choice of the almost complex $4$-manifold $(X, J)$ satisfying the conditions above and hence it is an invariant of the homotopy class of $\xi$, denoted by $\theta(\xi)$.  Note that $c^2_1 (X, J)$ is well-defined only when rational coefficients are used in (co)homology, which explains why $\theta(\xi) \in \mathbb{Q}$. According to \cite[Theorem 4.5]{Gompf98}, $\theta(\xi)$ depends only on $Y$ and the homotopy class of $\xi$, but independent of the orientation of $\xi$ and reverses sign if the orientation of $Y$ is reversed. The next result is well-known but we state it as a lemma to be able to refer to it in the rest of  the text. 

\begin{lemma} \label{lem: rhb}  Suppose that $Y$ is an oriented rational homology $3$-sphere equipped with a contact structure $\xi$. If $\theta(\xi) \neq -2$, then $(Y, \xi)$  does not admit a rational homology ball symplectic filling. \end{lemma} 

\begin{proof} Suppose that $(X, \omega)$ is a weak symplectic filling of $(Y, \xi)$. According to \cite[Section 5]{Etnyre98}, there is an almost complex structure $J$ on $X$ which is tamed by $\omega$ such that $(Y,\xi)$ is the strictly pseudoconvex boundary of $(X,J)$. In particular, $\xi$ is the complex tangencies of $J$ on $Y$. Now assume as an extra hypothesis that $X$ has the rational homology of the $4$-ball. This immediately implies that $\chi(X)=1$,  and also  $\sigma(X) =0$ and $c^2_1(X, J)=0$ since the intersection form on $X$ identically vanishes. We conclude that $\theta(\xi) =-2$. 
\end{proof}

Any contact structure on any closed oriented $3$-manifold can be conveniently described by a contact surgery diagram and the $\theta$-invariant of the underlying oriented $2$-plane field can be calculated form such a  diagram  as explained in \cite{DingGeigesStipsicz04} (see also  \cite[page 195]{OzbagciStipsicz2004}). 

\section{Comparing the $\theta$-invariants of contact structures} \label{sec: compare} 
In this section we will prove Proposition~\ref{canminimizes} that shows for certain Seifert fibered spaces the canonical contact structure $\xi_{can}$ minimizes the $\theta$-invariant among all tight contact structures. 

We begin with some preliminary observations about positive/negative-definite matrices and quadratic forms. 

\begin{lemma}\label{lem:negdef}
If $Q$ is a negative-definite matrix whose off-diagonal entries are non-negative, then $Q$ is inverse-negative, i.e., every entry of $Q^{-1}$ is negative. 
\end{lemma}
\begin{proof}
If we set $P=-Q$ then $P$ will be a positive-definite matrix whose off-diagonal terms are nonpositive. Such a matrix is called an $M$-matrix and it is known that its inverse has positive entries, see \cite[Chapter~6]{BermanPlemmons1994}. Thus we see the entries of $Q^{-1}$ are negative. 
\end{proof}

\begin{lemma} \label{lem: concave} Suppose that $Q$ is a negative-definite $m \times m$ matrix whose off-diagonal entries are nonnegative. Consider the quadratic form $f: \R^m \to \R$ given as $f({\bf x})={\bf x}^T Q^{-1} {\bf x}$ and let ${\bf y}$ be fixed nonzero vector in $\R^m$ whose entries are nonnegative.  If  we set $$D= \{{\bf x} \in \R^m \; | \;   |{\bf x}_i| \leq |{\bf y}_i| \;  \mbox{for all} \; 1 \leq i \leq m\},$$  then $f|_D$ attains its minimum at ${\bf y}$. Moreover,  for any  ${\bf x} \in D$, we have $f({\bf y}) < f({\bf x})$, provided that there exists some $1 \leq i \leq m$ such that  $|{\bf x}_i| <  |{\bf y}_i|.$
\end{lemma}

\begin{proof} Since $Q$ is negative-definite,  $Q^{-1}$ is also negative definite and hence $f$ is strictly concave. The set $D$ is a convex and compact subset of  $\R^m$. The strictly concave function $f|_D$  attains its minimum at one of the extreme points of the convex set $D$ (that is the vertices of the polygon $D$). Note that there are at most $2^m$ extreme  points of $D$, one of which is {\bf y}. The others are obtained by possibly negating the entries of ${\bf y}$.   It follows that $f$ restricted to $D$ attains  its minimum at ${\bf y}$ since all the entries of $Q^{-1}$ are negative by Lemma~\ref{lem:negdef}. The fact that all the entries of $Q^{-1}$ are negative, also implies the last statement of the lemma. 
\end{proof}

We are now ready to prove Proposition~\ref{canminimizes} that says if  $Y=Y(e_0;r_1,r_2,r_3)$ is a Seifert fibered space with $e_0\leq -3$ or $e_0=-2$ and $Y$ an $L$-space, then 
\[
\theta(\xi_{can})<\theta(\xi)
\]
for any contact structure $\xi$ on $Y$ not isotopic  to $\pm\xi_{can}$. 

\begin{proof}[Proof of Proposition~\ref{canminimizes}] Suppose that  $Y$ is as in the statement of the proposition. Then according to the classification of tight contact structures on $Y$ by Wu \cite{Wu06} in the case that $e_0\leq -3$ and by Ghiggini \cite{Ghiggini08} in the case that $e_0=-2$, all the contact structures on $Y$ come from Legendrian surgery on some Legendrian realization of the surgery diagram shown on the left of Figure~\ref{fig:e0lm2}. 

Let $\xi$ be any given tight (hence Stein fillable) contact structure on $Y$, which is represented by a Legendrian surgery diagram as described above.  By fixing a Legendrian surgery diagram for the contact structure $\xi$, we also fix a Stein surface $(X, J_\xi)$ inducing  $\xi$ on its boundary $Y$. Moreover, by a theorem due to Gompf \cite{Gompf98},  the first Chern class $c_1(X, J_\xi)$ is determined by the rotation numbers of the Legendrian unknots in the diagram. 

Next we observe that the inequality $\theta (\xi_{can}) < \theta(\xi)$ is equivalent to the inequality 
\[
c_1^2 (X, J_{\xi_{can}}) < c_1^2 (X, J_\xi),
\]
since the Euler characteristic and signature of the fillings of all the contact structures given from the Legendrian surgery diagram are the same. 
Let $Q_{X}$ denote the intersection matrix of the $4$-manifold $X$. Then $$c_1^2 (X, J_\xi) ={\bf r}^T_\xi  Q^{-1}_{X}  {\bf r}_\xi,  $$ where  ${\bf r}_\xi$ is the rotation vector in $\mathbb{R}^{m}$, whose entries correspond the rotation numbers of the Legendrian knots in the surgery diagram realizing $\xi$. 
 Then for any tight  contact structure $\xi$ on $Y$, which is not isotopic  to $\pm\xi_{can}$, the inequality 
\begin{equation}\label{eq: quad}
 {\bf r}_{\xi_{can}}^T  Q^{-1}_{X}  {\bf r}_{\xi_{can}}  <  {\bf r}^T_\xi  Q^{-1}_{X}  {\bf r}_\xi 
\end{equation} 
immediately follows from  Lemma~\ref{lem: concave} by setting $Q=Q_{X}$, ${\bf y}= {\bf r}_{\xi_{can}}$ and  ${\bf x}={\bf r}_{\xi}$. 
\end{proof}

\section{Seifert fibered structures on $S^1\times S^2$}\label{SFSonS1S2}
In this section we record a simple observation about which small Seifert fibered spaces can be obtained by surgery on a regular fiber in a Seifert fibration of $S^1\times S^2$ and the corresponding manifolds they bound. We also refer the reader to \cite{Matkovic2023} for the contact topology of $S^1\times S^2$ and constructing symplectic fillings of small Seifert fibered spaces from $S^1\times D^3$.
We begin with  the topology and contact geometry of $S^1\times S^2$ in Section~\ref{s1s2} and then discuss the symplectic fillings in Section~\ref{s1s2fillings}.

\subsection{The standard contact structure on $S^1\times S^2$}\label{s1s2}
We begin by proving that Figure~\ref{fig:s1s2} describes a Seifert fibered structure on $S^1\times S^2$, and that the regular fiber $F$ is a torus knot: 

\begin{figure}[htb]{
\begin{overpic}
{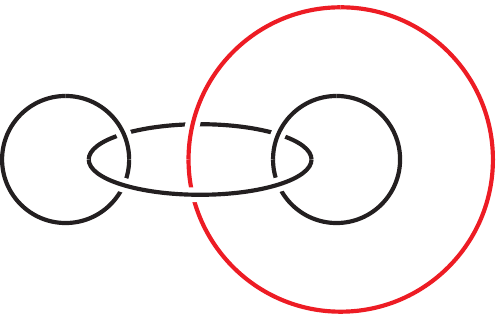}
\put(28, 110){$r$}
\put(152, 110){$-r$}
\put(210, 140){\color{dred}$F$}
\put(35, 67){$0$}
\end{overpic}}
 \caption{A surgery diagram for $S^1\times S^2$ with a Seifert fibered structure where $F$ is a regular fiber.}
  \label{fig:s1s2}
\end{figure}
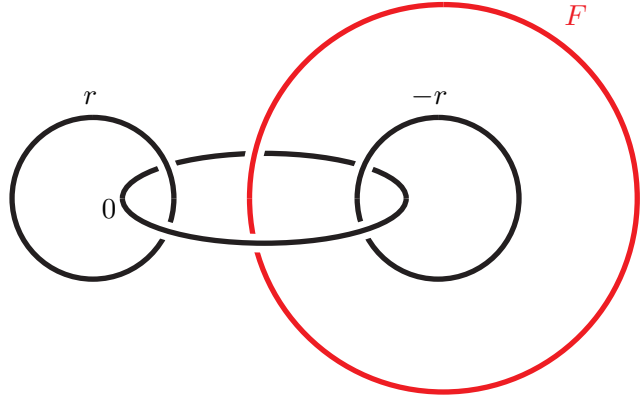
\begin{lemma}\label{sfsons1s2}
Surgery on the black unknots with coefficients $r, 0, -r$  in Figure~\ref{fig:s1s2} describes $S^1\times S^2$ with a Seifert fibered structure so that the regular fiber $F$ is a $(q, -q')$-torus knot if $r=q/p$ and $p'q-q'p=1$ with $p'\leq p$. 
\end{lemma}
\begin{proof}
The standard representation of $S^1\times S^2$ is given by $0$-surgery on the unknot.  This can equivalently be described as a quotient of $T^2\times[0,1]$ where the leaves of linear foliations of slope $\infty$ are collapsed on both $T^2\times \{0\}$ and $T^2\times \{1\}$. We can also describe $S^1\times S^2$ as  a quotient of $T^2\times[0,1]$ where the leaves of linear foliations of slope $q/p$ are collapsed on both $T^2\times \{0\}$ and $T^2\times \{1\}$. This second description of $S^1\times S^2$ is the one shown in Figure~\ref{fig:s1s2}. To see this notice that $T^2\times \{1/2\}$ splits $S^1\times S^2$ into two solid tori with lower and upper meridian $r$: $S_{r}$ and $S^{r}$ (see Section~\ref{ctstronsolidtori} for this notation and a discussion about solid tori with upper and lower meridians). This first solid torus $S_r$ is clearly just the result of $r$ surgery on the right-hand black unknot in Figure~\ref{fig:s1s2} but Dehn surgeries are not usually described with upper meridians so we need to convert $S^r$ into a solid torus with lower meridian. To this end notice that the map $f:  T^2\times[0,1]\to T^2\times[0,1]$ given by $f(\theta,\phi, t)=(-\theta, \phi, 1-t)$ induces a diffeomorphism from $S^r$ to $S_{-r}$. Thus the Dehn surgery corresponding to the second torus is shown on the left-most black curve in Figure~\ref{fig:s1s2}. 

We note the diffeomorphism of the torus represented  by the matrix
\[
\begin{bmatrix}
p'& p\\ 
q'&q
\end{bmatrix}
\]
gives a diffeomorphism from $T^2\times [0,1]$ to itself that induces a diffeomorphism from the first description of $S^1\times S^2$ above to the second description. The inverse diffeomorphism will take the fiber $F$ to a $(q,-q')$-curve on a Heegaard torus of $S^1\times S^2$. 
\end{proof}

Next, we examine in more detail the contact geometry of the various Seifert fibered structures on $S^1\times S^2$.  
\begin{lemma}\label{lem: framing}
Any realization of the standard tight contact structure on $S^1\times S^2$ via contact $r$ and $-r$ surgery on the two vertical black unknots in Figure~\ref{fig:s1s2} will have the basic slices of the same signs for one of the solid tori, and the other solid torus will have basic slices of the opposite sign. Moreover, the fiber $F$ in such a realization can be made Legendrian with contact framing agreeing with the framing coming from the Heegaard torus that contains it. 
\end{lemma}


\begin{proof}
Describing $S^1\times S^2$ as we did in the proof of Lemma~\ref{sfsons1s2}, the contact structure on $S^1\times S^2$ is simply $\ker (\cos \pi t\, d\theta + \sin \pi t\, d\phi)$ where $t$ is the coordinate on $[0,1]$ and $\theta$ and $\phi$ are angular coordinates on $T^2$. We can take the torus $T^2\times \{1/2\}$ and note it is linearly foliated by leaves of slope $0$. We can perturb this torus to be convex with two dividing curves of slope $0$. Then the torus splits $S^1\times S^2$ into two solid tori $N_0=S_\infty$ and $N_1=S^\infty$ and each torus has convex boundary and supports a unique tight contact structure. In particular, each torus is the standard neighborhood of a Legendrian knot, we denote them by $L_0$ and $L_1$, respectively. To think of $S^\infty$ as a neighborhood of $L_1$ in a more ``standard" way it should have a lower meridian, not an upper meridian. We can use the diffeomorphism $f$ in the proof of Lemma~\ref{sfsons1s2} to convert the upper meridian to a lower meridian. 
We note that this diffeomorphism will take tight contact structures on $S^\infty$ to tight contact structures on $S_\infty$ but the signs of the basic slices (that is, the decorations on the edges) will flip.

We can also describe $S^1\times S^2$ as  a quotient of $T^2\times[0,1]$, where the leaves of linear foliations of slope $r=q/p$ are collapsed on both $T^2\times \{0\}$ and $T^2\times \{1\}$. Note this description is the same as doing smooth $q/p$ surgery on $L_0$ and $-q/p$ surgery on $L_1$. Contact geometrically, we can think of these surgeries as removing the neighborhood $N_0$ of the Legendrian knot $L_0$ and replacing it with a solid torus $N_0'$ with lower meridian $q/p$ and then similarly replacing $N_1$ with $N_1'$. Each of the solid tori that are glued back in is determined by a path in the Farey graph (see Section~\ref{ctstronsolidtori}). We claim that the signs of the basic slices in $N_1'$ must all be the same and be opposite to the signs in $N_0'$ (which also must all be the same). If this were not the case, then when one adds the basic slices describing the contact structure on $N_1'$ to $N_0'$, they would flip signs, and hence we would have a path in the Farey graph with both signs, and this path will not be minimal. When the path is shortened to become minimal, we will have to shorten edges with mixed signs, and this implies the contact structure is overtwisted (see the discussion immediately after Theorem~\ref{thickenedT2}). Since the standard contact structure on $S^1\times S^2$ is tight, the signs must be as claimed. This means in any description of $S^1\times S^2$ in terms of contact surgery on $L_0$ and $L_1$, the stabilizations of $L_0$ must all be the same and opposite to those in $L_1$. 

Finally, we notice that $\partial N_1=\partial N_1'$ is a convex torus with vertical dividing curves. A Legendrian divide on this torus will be a realization of $F$ with contact twisting $0$ relative to the Heegaard framing. 
\end{proof}

We end this section by noting that any negative contact surgery on $F$ will result in a small Seifert fibered space with $e_0\geq -1$. 

\begin{lemma}\label{onlye0}
Any small Seifert fibered space obtained from the standard contact structure on $S^1\times S^2$ by some negative contact surgery on a Legendrian realization of a fiber in a Seifert fibration of $S^1\times S^2$ will have $e_0\geq -1$. 
\end{lemma}
\begin{proof}
The Seifert fibration shown in Figure~\ref{fig:s1s2} is not in standard form since one, or both, of $r$ and $-r$ is not less than $-1$. If we assume $r$ is positive and it takes $n+1$ left-handed Rolfsen twists to convert $r$ into a number less than $-1$, then it will take $n$ right-handed Rolfsen twists to convert $-r$ into a number less than $-1$. To see this assume that $r=q/p\in(1/(n+1),1/n)$, for some non-negative integer $n$, then one may easily check that $q/(p-(n+1)q)<-1$ and similarly $-q/(p-nq)<-1$. The first surgery coefficient is obtained from $q/p$ by $n+1$ left-handed Rolfsen twists, and this reduces the framing on the center curve in Figure~\ref{fig:s1s2} by $n+1$. The second surgery coefficient is obtained from $-q/p$ by $n$ right-handed Rolfsen twist, and this increases the framing on the center curve by $n$. Thus when we have normalized these two surgery coefficients we see the framing on the central curve is $-1$. Here we assume that $r\not=1/n$, since otherwise  the surgery diagram describes a Seifert fibered structure on $S^1\times S^2$ with no singular fibers and hence surgery on $F$ will produce a Seifert fibered space with only one singular fiber (that is a lens space) and we are interested in constructing Seifert fibered spaces with three  singular fibers. 

 In Lemma~\ref{lem: framing}, we showed that the regular fiber $F$ can be realized by a Legendrian knot with contact framing agreeing with the framing coming from the Heegaard torus (which is also the same as the framing coming from the fibration).
In Figure~\ref{fig:s1s2} that will be the zero framing. Thus any negative contact surgery on the fiber will have smooth surgery coefficient less than $0$. When the surgery coefficient is less than $-1$, then the Seifert fibered space will have $e_0=-1$. More generally, if the surgery coefficient is in $(-1/n, -1/(n+1))$, for some positive integer $n$, then $e_0=n-1$. (When the surgery coefficient is $-1/n$ for some $n$, then we obtain a lens space, and do not consider those.)
\end{proof}

\begin{remark}
The same argument as in the proof clearly shows that if we perform a surgery on $F$ with coefficient $q/p\in(1/(n+1),1/n)$, for $n\geq 0$,  then we will obtain a Seifert fibered space with $e_0=-1-n$.
\end{remark}

\subsection{Symplectic fillings}\label{s1s2fillings}

It is interesting to build contact structures on the above Seifert fibered spaces that can bound symplectical rational homology balls.
To do so we set up some new notation for Seifert fibered spaces. We denote by $S(r_1,r_2,r_3)$ the result of surgery on the link in Figure~\ref{fig:s1s2} where the surgery coefficients on the three vertical circles are $-1/r_1,-1/r_2,-1/ r_3$ and the surgery coefficient on the horizontal circle is $0$. This is the ``unnormalized" form of a Seifert fibered space (note that by performing a right-handed Rolfsen twist on one of the vertical curves and a left-handed twist on another gives another unnormalized presentation of the same manifold). Of course one changes this surgery diagram to a normalized form by Rolfsen twisting the vertical curves until their surgery coefficients are all less than $-1$. 

\begin{lemma}\label{cable}
For any fraction $q/p$, relatively prime positive integers $h<m$ , and a non-positive integer
$-k$ , the Seifert fibered space
\[
S\left(\frac pq, -\frac pq, \frac{m^2}{km^2+mh+1}\right)
\]
admits a tight contact structure that is symplectically filled by a rational homology ball. This symplectic filling is built by attaching a Weinstein $2$-handle to $S^1\times D^3$ along a negative Legendrian cable of a Legendrian torus knot  (determined by $r=q/p$ as in Lemma~\ref{sfsons1s2}). 
\end{lemma}
\begin{remark}\label{exoffillings}
We note that one can obtain any $e_0\geq -1$ via this construction. Indeed, if $-k<-1$ and $q/p\in (1/(n+1),1/n)$ then the normalized form of this Seifert fibered space is 
\[
Y\left(-1;\frac{(n+1)q-p}{q}, \frac{p-nq}{q}, \frac{m^2}{km^2+mh+1}\right)
\] 
If $k=0$ then the Seifert fibered space will have $e_0>-1$.  For example, in the above lemma if we take $k=0$ and $h=1$ then the resulting Seifert fibered space is 
\[
Y\left(m-2;\frac{(n+1)q-p}{q}, \frac{p-nq}{q}, \frac{1}{m+1}\right).
\]
To see this notice that the surgery coefficient on $F$ is $-\frac{m+1}{m^2}$ and $m-1$ right-handed Rolfsen twist will make the coefficient $-m-1$. 
\end{remark}

\begin{proof}
As discussed in the proof of Lemma~\ref{sfsons1s2}, we see the surgery diagram in Figure~\ref{fig:s1s2} with $r=q/p$ gives $S^1\times S^2$ and the red curve, which we denote by $F$, is a torus knot which is a regular fiber in a Seifert fibered structure on $S^1\times S^2$. We note that $F$ is also a leaf in the characteristic foliation for a Heegaard torus for $S^1\times S^2$ and so its contact framing is $0$. Thus, the contact surgery coefficient, when considered below, is the same as the smooth surgery coefficient.

We now claim that a $-\frac{km^2+mh+1}{m^2}$ surgery on $F$, which will be the Seifert fibered space $S(p/q,-p/q, m^2/(km^2+mh+1)$, will bound a symplectic rational homology ball.  We will construct the symplectic rational homology ball by attaching a Weinstein $2$-handle to the Weinstein manifold $S^1\times D^3$ along a Legendrian knot in its boundary $S^1\times S^2$. To see the smooth knot type that will produce the desired Seifert fibered space, we recall that Gordon \cite[Section~7]{Gordon1983} showed that $-mh-1$ surgery on the $(m,-h)$-cable of a knot $L$ is equivalent to $\frac{-mh-1}{m^2}$ surgery on $L$. Below we will discuss how to perform such a surgery on a Legendrian cable of a stabilization of the Legendrian torus knot $F$.

We recall the notion of a standard negative cable of a Legendrian knot $L$. As before $S^{\pm}(L)$ denotes positive/negative stabilization of $L$. Let $0<h<m$ be coprime. Choose a standard neighborhood $N$ of $L$ and $N_\pm$ a standard neighborhood of $S^{\pm}(L)$. Taking the framing on $L$ to be the contact framing we see that $N\setminus N_\pm$ is $T^2\times [0,1]$   and within $N\setminus N_{\pm}$ there exist convex tori with any dividing slope in $[-1,0]$. Since $-h/m\in[-1,0]$, we can realize $L^\pm_{(m,-h)}$ as a Legendrian divide on a convex torus in $N\setminus N_\pm$. The Legendrian $L^\pm_{(m,-h)}$ is called the \dfn{$\pm$-standard $(m,-h)$-cable of $L$} and the contact framing of $L^\pm_{(m,-h)}$ is $0$ relative to the torus on which it lies. Finally, if $tb(L)=-k$, then  $L^\pm_{(m,-h)}$ is smoothly the $(m, -h-km)$-cable of $L$. 

Now since $F$ is Legendrian with contact framing $0$, we can stabilize it $k$ times to get a Legendrian knot $\widehat F$ with Thurston-Bennequin invariant $-k$. By Gordon's result mentioned above, performing Legendrian surgery on $\widehat F^\pm_{(m,-h)}$ will be the same as $-\frac{km^2+mh+1}{m^2}$ {surgery} on $F$. Thus we obtain the desired symplectic rational homology ball by attaching a Weinstein $2$-handle to $S^1\times D^3$ along the Legendrian knot $\widehat F^\pm_{(m,-h)}$ in its boundary. \end{proof}

\begin{remark}\label{nege0smooth}
If we are only interested in building smooth rational homology balls, notice that we can perform the same construction, but we can use any framing on $F$. In terms of Lemma~\ref{cable}, this would correspond to taking $k$ to be any integer. In this case, we note that when $-k$ is a positive integer we will be able to get Seifert fibered spaces with $e_0$ any negative integer. For example, if $-k=1$ and $h=m-1$ then the above construction creates a smooth rational homology ball filling of 
\[
Y\left(-m-3;\frac{(n+1)q-p}{q}, \frac{p-nq}{q}, \frac{m-2}{m-1}\right).
\]
Of course, Theorem~\ref{main1} says this rational homology ball cannot symplectically fill any contact structure on the manifold. 
\end{remark}
We can now prove Theorem~\ref{arbitmany} that notes when $e_0=-1$ there can be arbitrarily many tight contact structures admitting rational homology ball fillings. 

\begin{proof}[Proof of Theorem~\ref{arbitmany}]
We note that when $k$ is large in the above proof the Legendrian knot $F$ must be stabilized many times before the cable is done. Thus the cable will have $2(k+1)$ possible rotation numbers (because there are $k+1$ rotation numbers for the stabilized $F$ and for each one of these, the $\pm$-cable will have distinct rotation numbers). By \cite{LiscaMatic97} we know that performing Legendrian surgery on these cables will result in distinct contact structures. 

Recall when $k>0$, we will always obtain a manifold with $e_0=-1$. Thus we see that there are examples of $e_0=-1$ small Seifert fibered spaces admitting arbitrarily many distinct contact structures with symplectic rational homology ball fillings. 
\end{proof}

We would like to reinterpret the above construction to better see how it relates to Christian and Menke's splitting theorem, Theorem~\ref{splittingthm}. To this end, we first note that if $L$ is a Legendrian knot in $(Y,\xi)$ and $L'$ is a Legendrian knot in $(Y',\xi')$, then attaching a round $1$-handle to a $4$-manifold with boundary $Y\cup Y'$ along $L$ and $L'$ changes the boundary as follows: a standard neighborhood of $L$ is removed from $Y$, a standard neighborhood of $L'$ is removed $Y'$, and the resulting boundary components are glued together so that the meridian and dividing curves on one of the tori are mapped to the meridian and dividing curves on the other torus. We will call this process {\em attaching a round $1$-handle to $Y\cup Y'$}, even though technically we need a $4$-manifold with boundary $Y\cup Y'$ to perform the operation. Since our interest lies in the induced effect on the  $3$-manifolds, we suppress mention of the $4$-manifold until it is needed.

A useful example of this construction is when $Y'$ is $S^1\times S^2$ and $L'$ is the Legendrian knot in the knot type of $S^1\times \{pt\}$ whose neighborhood has dividing curves of slope $0$. We note that when we perform the round $1$-handle attachment along $L$ and $L'$,  we simply remove a neighborhood of $L$ from $Y$ but then reglue a solid torus that replaces that neighborhood. That is, we are once again left with $(Y,\xi)$. 

We are now ready to reinterpret the construction in the proof of Theorem~\ref{arbitmany} in terms of round a $1$-handle attachment.  Recall that only the lens spaces $L(m^2,mh-1)$ symplectically bound a rational homology ball, which we denote by $B_{m,h}$, and only their universally tight contact structure bounds such a symplectic manifold, see \cite{ChristianLi23, EtnyreRoy21,Lisca08}. The manifold $B_{m,h}$ is built by attaching a Weinstein $2$-handle to a Legendrian $(m,-h)$-torus knot $L''$ in $S^1\times S^2=\partial (S^1\times D^3)$. Let $L'$ be a Legendrian representative of $S^1\times \{pt\}$ from the previous paragraph. We can assume that $L''$ is outside the standard neighborhood of $L'$. Thus if we perform a round $1$-handle attachment to $Y\cup S^1\times S^2$ along $L$ in $Y$ and $L'$ in $S^1\times S^2$ we obtain $Y$, as noted in the previous paragraph. However, we notice that $L''$ becomes the $(m,-h)$-cable of $L$ in $Y$ after the handle attachment. Thus if we first do $-mh-1$ surgery on $L''$ in $S^1\times S^2$ and then perform the above round $1$-handle attachment, this will result in the manifold obtained from $Y$ by doing the same surgery on the $(m,-h)$-cable of $L$. 

\begin{lemma}
The symplectic rational homology ball constructed in Lemma~\ref{cable} can also be obtained by attaching a round Weinstein $1$-handle to $(S^1\times D^3)\cup B_{m,h}$ along a Legendrian  torus knot in $S^1\times S^2$ (as described in Lemma~\ref{cable}) and a Legendrian realization of the core of a Heegaard torus in $L(m^2,mh-1)$. 
\end{lemma}

\begin{proof}
Using the notation from the discussion above we can think of $L'$ as a Legendrian realization of the core of a Heegard torus in $L(m^2,mh-1)=\partial B_{m,h}$. If we take the Legendrian torus knot in $S^1\times S^2=\partial (S^1\times D^3)$ from the proof of Lemma~\ref{cable}, then the discussion above shows that the new symplectic $4$-manifold $X$, obtained by adding a Weinstein round $1$-handle to $B_{m,h} \cup S^1\times D^3$, has boundary obtained by doing Legendrian surgery to the $(m,-h)$-cable of the torus knot. Moreover, recall a round $1$-handle is a union of a $1$-handle and a $2$-handle. The $2$-handle will cancel the $1$-handle in $B_{m,h}$ and the $1$-handle can be canceled against the $0$-handles in $B_{m,h}$. Thus all that is left of $B_{m,h}$ is its $2$-handle which will now be attached to the $(m,-h)$-cable of the torus knot. Thus $X$ is the same symplectic rational homology ball constructed in Lemma~\ref{cable}.
\end{proof}

We illustrate the previous proof in an example. In Figure~\ref{fig:extorussurgery} the left three black curves describe $S^1\times S^2$ with $r=3$, and according to Lemma~\ref{sfsons1s2} the red curve (if you ignore the right two black curves) is a $(3,-2)$-torus knot. 
\begin{figure}[htb]{
\begin{overpic}
{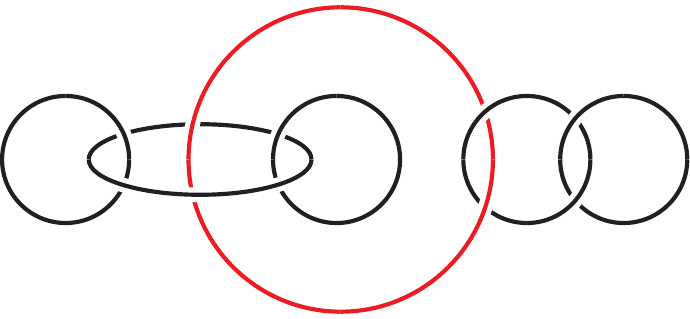}
\put(21, 112){$-3$}
\put(149, 112){$-3/2$}
\put(210, 140){$-4$}
\put(244, 112){$-2$}
\put(292, 112){$-5$}
\put(27, 67){$-1$}
\end{overpic}}
 \caption{A surgery diagram for $Y\left(-1;\frac 13,\frac 23,\frac 9{31}\right)$.}
  \label{fig:extorussurgery}
\end{figure}
The right two black curves describe a lens space $L(9,5)$. Combining the left black curves with the right black curves using the red curve is the result of a round $1$-handle attached to the $(3, -2)$-torus knot in $S^1\times S^2$ and the core of a Heegaard torus for $L(9,5)$. We easily see this is the Seifert fibered space $Y\left(-1;\frac 13,\frac 23,\frac 9{31}\right)$.

To see the symplectic rational homology ball filling of $Y\left(-1;\frac 13,\frac 23,\frac 9{31}\right)$ we first consider Figure~\ref{fig:pieces}.
\begin{figure}[htb]{
\begin{overpic}
{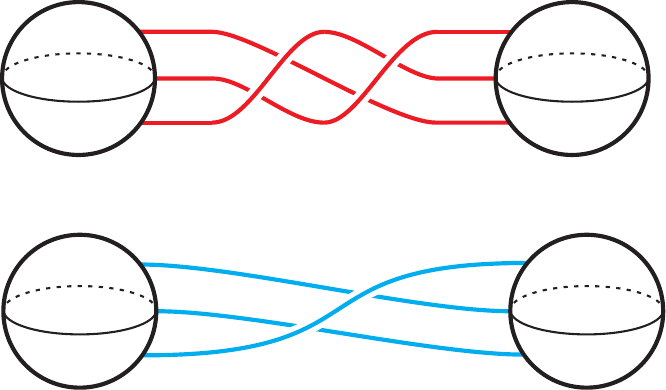}
\put(180, 10){$-4$}
\end{overpic}}
 \caption{The top is $S^1\times S^2$ with the $(3,-2)$-torus knot shown in red. The bottom is the rational homology ball with boundary $L(9,5)$.}
  \label{fig:pieces}
\end{figure}
 Here we see $S^1\times S^2$ with the torus knot $(3,-2)$-shown in red (this is the left-most four knots shown in Figure~\ref{fig:extorussurgery} (but without the surgery performed on the red curve). In the figure we also see the rational homology ball bounded by $L(9,5)$. In Figure~\ref{fig:finaex} we see the result of attaching a round $1$-handle to $S^1\times D^3$ and the rational homology ball $L(9,5)$ bounds. 
\begin{figure}[htb]{
\begin{overpic}
{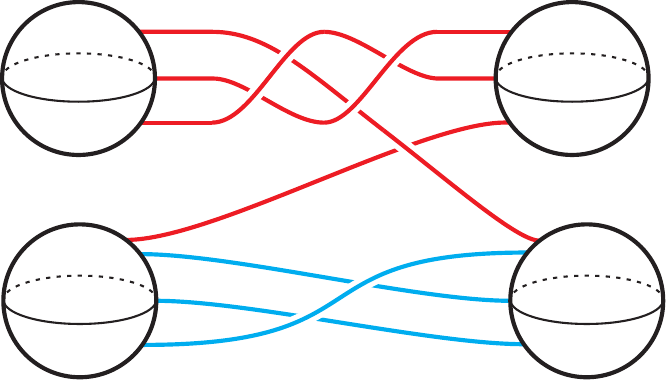}
\put(180, 10){$-4$}
\put(180, 167){$-3$}
\end{overpic}}
 \caption{The rational homology ball bounded by $Y\left(-1;\frac 13,\frac 23,\frac 9{31}\right)$.}
  \label{fig:rhbex}
\end{figure}
We note that in this figure the $-3$ framing on the red curve is with respect to the torus the red curve sits on. 

To understand this framing,  we note that in Figure~\ref{fig:extorussurgery} the $-4$ framing on the red curve is with respect to the torus framing the red curves sits on as a torus knot. Thus for this to be a Legendrian surgery the Legendrian knot will have Thurston-Bennequin invariant $-3$. When we add the round $1$-handle, the Legendrian $(3,-2)$-torus knot is connect summed with a Thurston-Bennequin invariant $0$ core of the lens space. This results, as $\tb(L\#L')=\tb(L)+\tb(L')+1$, in a Thurston-Bennequin invariant $-2$ curve on which Legendrian surgery is performed. Thus the framing is $-3$, with respect to the framing on the curve coming from the Heegaard torus in $S^1\times S^2.$

Now, if we cancel the red $2$-handle in Figure~\ref{fig:rhbex} (after sliding the blue $2$-handle over it $3$ times), we will obtain Figure~\ref{fig:finaex}. 
\begin{figure}[htb]{
\begin{overpic}
{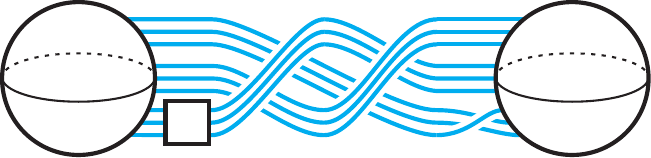}
\put(82, 13.5){$-3$}
\put(185, 1){$-31$}
\end{overpic}}
 \caption{The rational homology ball bounded by $Y\left(-1;\frac 13,\frac 23,\frac 9{31}\right)$.}
  \label{fig:finaex}
\end{figure}
This is clearly $-31$ surgery on the $(3,-10)$-cable of the $(3,-2)$-torus knot in $S^1\times S^2$ (note the cable coefficients are with respect to the Heegaard torus framing, with respect to the "Seifert framing" it is the $(3, -28)$-cable, but to agree with our discussion above, we need to use the Heegaard torus framing.) We end by noting that according to Gordon's result mentioned in the proof of Lemma~\ref{cable} we see that is the same as $-31/9$ surgery on the $(3,-2)$-torus knot in $S^1\times S^2$, which by Lemma~\ref{cable} is precisely $Y\left(-1;\frac 13,\frac 23,\frac 9{31}\right)$. We also mention that Golla and Starkston in \cite{GollaStarkston2024} provided some specific such examples with direct Kirby calculus methods.

\begin{remark}
We note that the reverse process of attaching a round $1$-handle is splitting the $4$-manifold along an $S^1\times D^2$, and the boundary will be an $e$-splitting of the boundary of the $4$-manifold. This $e$-splitting will result in $S^1\times S^2$ and $L(m^2,mh-1)$. 
\end{remark}

\section{Seifert fibered spaces with $e_0\leq -3$}\label{Mneg}
In this section we prove one of our main theorems, namely Theorem~\ref{main1}, concerning symplectic rational homology ball fillings of $Y(e_0;r_1,r_2,r_3)$ with $e_0\leq -3$. 

Recall from the introduction that any tight contact structure on these manifolds comes from Legendrian surgery on a Legendrian link realizing the link in Figure~\ref{fig:plumbing}. In Figure~\ref{fig:e0lm2} we give two other surgery presentations of the Seifert fibered space that we will need below. They are all related by slam dunk operations.

\begin{figure}[htb]{
\begin{overpic}
{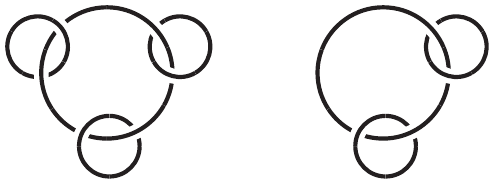}
\put(-23, 80){$-1/r_1$}
\put(95, 80){$-1/r_2$}
\put(68, 7){$-1/r_3$}
\put(13, 25){$e_0$}
\put(232, 80){$-1/r_2$}
\put(200, 5){$-1/r_3$}
\put(125, 25){$e_0+r_1$}
\end{overpic}}
\caption{Two surgery diagrams for $Y(e_0;r_1,r_2,r_3)$.}
\label{fig:e0lm2}
\end{figure}
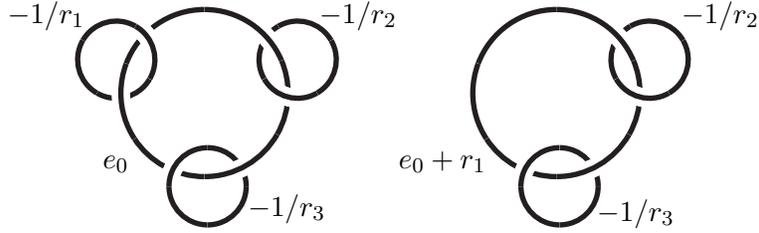

We begin with a preliminary result. Recall from Definition~\ref{m-consistent} that a contact structure on $Y(e_0;r_1,r_2,r_3)$ is called consistent if all the Legendrian knots in a Legendrian surgery diagram realizing Figure~\ref{fig:plumbing} for the contact structure are stabilized in the same way. If a contact structure is not consistent, we call it \dfn{inconsistent}. 

\begin{proposition}\label{inconsistent1}
If $\xi$ is an inconsistent contact structure on $Y=Y(e_0;r_1,r_2,r_3)$ with $e_0\leq -3$, then $(Y, \xi)$   is not filled by a symplectic rational homology ball. 
\end{proposition}
\begin{proof}
Suppose $(Y(e_0;r_1,r_2,r_3),\xi)$ has a rational homology ball symplectic filling $X$. 
We show below that there is a mixed torus $T$ such that any $e$-splitting associated to $T$ will result in either (1) $S^3$ and the connected sum of three lens spaces, (2) $S^3$ and the connected sum of a lens space and another Seifert fibered space $Y(e_0;r_1', r'_2, r'_3)$ where two of the $r_i'$ agree with $r_i$, (3) a lens space and another Seifert fibered space as in (2), or (4) two lens spaces. 
 By Remark~\ref{disconnected} we see that Theorem~\ref{splittingthm} gives fillings $X_1$ and $X_2$ of the components of the $e$-splitting. Since both of the components of the $e$-splitting are rational homology spheres, a simple homology computation shows that neither can bound a rational homology $S^1$. Thus by Remark~\ref{homologyofsplitting} the filling $X$ could not have existed. 

We are assuming that $\xi$ is not consistent. The first way this can happen is if the central curve is stabilized both positively and negatively. Thus there is a Legendrian $L$ realizing the central curve such that $\xi$ is obtained by surgery on a Legendrian link where the central curve is $L'=S_+(S_-(L))$. If we consider the neighborhood $N$ of $L$, then Legendrian surgery on $L'$ will remove a neighborhood $N'$ of $L'$and glue in a solid torus $S$ with meridional slope $\tb(L')-1$. Thus we have a mixed torus $T$ in $N-N'$ with slope $\tb(L)-1$ and the slopes of the basic slices showing that $T$ is a mixed torus are $\tb(L)$ and $\tb(L')$. Let $S'$ be the solid torus in $Y(e_0;r_1,r_2,r_3)$ that $T$ bounds. 
We can easily see that the only slope in $E_T$ is $\infty$. Thus the $e$-splitting along $T$ will glue a solid torus with meridional slope $\infty$ to $S'$, which has meridional slope $\tb(L')-1$. This gives $S^3$. The other piece of the splitting is the result of removing $S'$ from $Y(e_0;r_1,r_2,r_3)$ and gluing in a solid torus with slope $\infty$. That is just removing $L'$ from the Legendrian surgery diagram. This results in a connected sum of three lens spaces. This gives Case~(1) above. 

The second way $\xi$ can be inconsistent is if the signs of the stabilizations along one of the components belonging to any of the legs are inconsistent. Thus, as discussed in the previous paragraph, we will have a mixed torus associated with this component, and the mixed torus will split the Seifert fibered space into $S^3$ and the result of removing this component from the surgery diagram. That is, it will be 
\begin{itemize}
\item the connected sum of a lens space and a Seifert fibered space, or
\item if the inconsistently stabilized knot is adjacent to the central curve, we could get the connected sum of
\begin{itemize}
\item two lens spaces or 
\item a lens space and $S^1\times S^2$. 
\end{itemize}
\end{itemize}
The last case cannot happen by Lemma~\ref{onlye0} since $e_0<-1$. The other two cases give Case~(2) or~(4) above. 

The third way $\xi$ can be inconsistent is if the stabilizations on each component of a leg are consistent but some of the components are stabilized differently. This means if we think of $\xi$ as coming from contact surgery on the link in Figure~\ref{fig:e0lm2}, then the contact structure on one of the solid tori, associated to the knot $K$, describing the contact surgery corresponds to a partially decorated path where the signs in a given continued fraction block are the same, but the signs between some of the continued fraction blocks differ. Thus we have a mixed torus $T$ at the juncture between two continued fraction blocks. As discussed in Remark~\ref{splits} we see that any $e$-splitting will consist of a lens space and 
a Seifert fibered space as described above. More specifically, the $e_0$ of the new Seifert fibered space is the same as the original Seifert fibered space. To see this, we need to know that the new $r_i'$ is between $0$ and $1$, which is equivalent to $-1/r'_i<-1$, but this follows since all the exceptional slopes $E_T$ are less than $-1$. This gives Case~(3) above. (We might have part of the splitting contain $S^1\times S^2$, but as above, this cannot happen unless $e_0=-1$.)

The last way that $\xi$ can be inconsistent is if the stabilization on the central curve and the stabilization on one of the legs, say the first leg, are different. In this case, consider the surgery diagram on the right-hand side of Figure~\ref{fig:e0lm2}. The contact surgery on the middle curve will now be given by a contact structure on a solid torus corresponding to a path whose first continued fraction block has one sign and all the others have the opposite sign. It is not hard to see that the $e$-splitting in this case yields either Case~(3) or Case~(4) above. (We might have part of the splitting contain $S^1\times S^2$, but as above, this cannot happen unless $e_0=-1$.)
\end{proof}

We next consider consistent contact structures. 
\begin{proposition}\label{secondpartof1}
If $\xi$ is a consistent contact structure on $Y=Y(e_0;r_1,r_2,r_3)$ with $e_0\leq -3$, then $(Y, \xi)$  admits a symplectic rational homology ball filling if and only if $Y$ belongs to one of the three families of Seifert fibered spaces depicted in Figure~\ref{fig:plumbing1}.
\end{proposition}

\begin{proof}
As discussed before the statement of Theorem~\ref{main1}, when $e_0 \leq -3$,  the consistent contact structure on $Y=Y(e_0;r_1,r_2,r_3)$ is the Milnor fillable contact structure $ \xi_{can}$, and also recall from the introduction that the contact $3$-manifold  $(Y, \xi_{can})$ admits a symplectic rational homology ball filling if and only if the minimal good resolution graph of the corresponding singularity (a negative definite star-shaped tree with three legs) belongs to one of the three infinite families shown in Figure~\ref{fig:plumbing1}. 
\end{proof}

The above two propositions prove Theorem~\ref{main1} that tells us, when $e_0\leq -3$, precisely which contact structures on $Y(e_0;r_1,r_2,r_3)$ admit symplectic rational homology ball fillings. 
\begin{proof}[Proof of Theorem~\ref{main1}]
Proposition~\ref{inconsistent1} tells us that an inconsistent contact structure cannot be filled by a symplectic rational homology ball and Proposition~\ref{secondpartof1} addresses the case of consistent contact structures. 
\end{proof}

We now turn to the proof of Theorem~\ref{lagslice} which says that when $e_0\leq -3$, certain Legendrian knots in $Y(e_0;r_1,r_2,r_3)$ that are Lagrangian slice in the symplectic filling obtained by attaching Weinstein $2$-handles to the Legendrian realization of Figure~\ref{fig:plumbing} where all the Legendrian unknots have been stabilized consistent, but are not Lagrangian slice in their rational homology ball fillings. 
\begin{proof}[Proof of Theorem~\ref{lagslice}]
let $\xi$ be a tight contact structure on $Y(e_0;r_1,r_2,r_3)$, with $e_0\leq -3$, obtained from a plumbing graph as in Figure~\ref{fig:plumbing}. The meridians to each of the $2$-handles in Figure~\ref{fig:plumbing} bound Lagrangian disks in the symplectic filling corresponding to the plumbing diagram since they are belt spheres to the Stein handles. But suppose that $L$ bounded a Lagrangian disk in a rational homology ball filling of $\xi$. Then we could remove a neighborhood of the Lagrangian disk to obtain a filling of the contact $3$-manifold described by removing the $2$-handle corresponding to the Legendrian knot. This would give a rational homology $S^1\times D^3$ filling of a connected sum of lens spaces or a lens space and a small Seifert fibered space with $e_0\leq -3$. Since that cannot happen, as discussed in Remark~\ref{homologyofsplitting}, we see that $L$ cannot be Lagrangian slice in the rational homology ball filling of $\xi$.
\end{proof}

\section{Seifert fibered spaces with $e_0=-2$}\label{Lnegm2}
In this section we will prove Theorem~\ref{main2} concerning symplectic rational homology ball fillings of $Y(-2;r_1,r_2,r_3)$ and its Corollary~\ref{possiblefillingsfore0m2}. We begin by observing that for any small Seifert fibered space $Y(-2;r_1,r_2,r_3)$ in $\mathcal{QHB}$, that is any small Seifert fibered space that belongs to one of the seven infinite families depicted in Figure~\ref{fig:plumbing2}, only the canonical contact structure has a symplectic rational homology ball filling. 

\begin{proposition}\label{posfillsm3}
If $\xi$ is a tight contact structure  on $Y=Y(-2;r_1,r_2,r_3) \in \mathcal{QHB}$,  then $(Y, \xi)$ admits a symplectic rational homology ball filling if and only if $\xi$ is contactomorphic to $\xi_{can}$. 
\end{proposition}
\begin{proof}
We know from \cite{BhupalStipsicz11} that the contact structure $\xi_{can}$ for any $Y\in \mathcal{QHB}$ admits a symplectic rational homology ball filling. If $\xi$ is any contact structure not equal to $\pm \xi_{can}$, then according to Proposition~\ref{canminimizes} we know that $\theta(\xi)>\theta(\xi_{can})=-2$ and thus by Lemma~\ref{lem: rhb} it cannot be symplectically filled by a rational homology ball. 
\end{proof}

We now consider the possible contact structures on an $L$-space of the form $Y(-2;r_1,r_2,r_3)$ that can be filled by a symplectic rational homology ball.

\begin{proposition}\label{posfillsm2}
If $Y=Y(-2;r_1,r_2,r_3)$ does not belong to $\mathcal{QHB}$, but it is an $L$-space, and $\xi$ is any contact structure on $Y$, so that $(Y, \xi)$ admits a symplectic rational homology ball filling, then $\xi$ must be consistent or mostly consistent.
\end{proposition}
\begin{proof}
Given a contact structure $\xi$ on $Y$ that is not consistent or mostly consistent, one of the legs in the surgery diagram will contain Legendrian knots with opposite stabilizations and hence there will be a mixed torus in the solid torus corresponding to that leg. Arguing as in the proof of Proposition~\ref{inconsistent1} we see that an $e$-splitting of $Y$ will result in either a small Seifert fibered space with $e_0=-2$ and a lens space, or two lens spaces (and none of the mentioned lens spaces can be $S^1\times S^2$). Moreover, Aceto \cite[Theorem~5.5]{Aceto2020} proved that no Seifert fibered space with three singular fibers can smoothly  bound a homology $S^1\times D^3$ and we already observed that no lens spaces can bound such a space. Thus following Remark~\ref{homologyofsplitting} we see that $\xi$ cannot symplectically bound a rational homology ball. 
\end{proof}

Our second main theorem follows.
\begin{proof}[Proof of Theorem~\ref{main2}]
The statement of Theorem~\ref{main2} is precisely the combination of Proposition~\ref{posfillsm3} and Proposition~\ref{posfillsm2}. 
\end{proof}
Finally, we give the proof of Corollary~\ref{possiblefillingsfore0m2}: if  $Y=Y(-2;r_1,r_2,r_3)$ is an $L$-space, then, up to isomorphism,  there are at most four contact structures  $\xi$ on $Y$ so that $(Y, \xi)$ admits a  symplectic rational homology ball filling. 

\begin{proof}[Proof of Corollary~\ref{possiblefillingsfore0m2}]
From Proposition 6.2 we know that any contact structure 
 $\xi$ on $Y$ that admits a symplectic rational homology ball filling is either consistent or mostly consistent. Equivalently, along the three legs the stabilizations are either all of the same sign, or two have one sign while the third has the opposite sign. Thus, up to isomorphism,  there are at most four stabilization patterns in total: the consistent pattern (all legs stabilized the same way) and the three mostly consistent patterns determined by which leg carries the opposite sign.
\end{proof}

\section{Seifert fibered spaces with $e_0\geq 0$}\label{pos}
In this section we prove Theorem~\ref{main3} concerning symplectic fillings of small Seifert fibered spaces $Y(e_0;r_1,r_2,r_3)$ with $e_0\geq 0$  and its corollary stated in the introduction. Recall that Theorem~\ref{main3} asserts the following:
if a contact structure on $Y(e_0;r_1,r_2,r_3)$ with $e_0\geq 0$ is filled by a symplectic rational homology ball, then it must be mostly consistent. In particular, if the contact structure is consistent, then it has no symplectic rational homology ball filling.

\begin{proof}[Proof of Theorem~\ref{main3}]
We first note that if one of the unknots in the Legendrian surgery description of $Y(e_0;r_1,r_2,r_3)$ with $e_0\geq 0$, shown at the bottom of Figure~\ref{fig:e0gqe0} is stabilized inconsistently, or one of the legs is stabilized inconsistently, then we can argue as in the proof of Proposition~\ref{inconsistent1} to show that there is no symplectic rational homology ball filling. The only difference, in this case, is to note that if one of the parts of the $e$-splitting is still a Seifert fibered space with three singular fibers, then $e_0$ is still greater than or equal to $0$ and hence cannot bound a rational homology $S^1$ and if there are only two singular fibers, then we either have a lens space, which cannot bound a rational homology $S^1$, or $S^1\times S^2$ which does bound $S^1\times D^3$. In this last case, the filling will come from attaching a symplectic round $1$-handle to $S^1\times D^3$ union a rational homology ball bounding a lens space. The round $1$-handle is attached to an $(a,b)$-torus knot in $S^1\times S^2$ and the core of a Heegaard torus for the lens space. When the round $1$-handle is attached, the $3$-manifold changes by removing a neighborhood of the $(a,b)$-torus knot and replacing it with the solid torus obtained by removing a neighborhood of the core of the Heegaard torus in the lens space. We can see in the proof of Lemma~\ref{cable}, see the discussion in Remark~\ref{exoffillings}, that for this construction to yield a Seifert fibered space with $e_0>-1$, the $(a,b)$-torus knot cannot be stabilized; moreover since the contact structures on lens space that have rational homology ball fillings are universally tight, the latter solid torus is universally tight. Therefore, the leg in the surgery diagram for this solid torus is, in fact, consistent, contradicting our assumption that it was the inconsistent leg. It follows that if any leg is inconsistent, there is no symplectic rational homology ball filling. Hence, for $e_0\ge 0$, any contact structure on $Y(e_0;r_1,r_2,r_3)$ that admits such a filling must be either consistent or mostly consistent.

In the remainder of the proof, we rule out the consistent case. Each leg is stabilized consistently, and all three legs are stabilized in the same way. In this setting, we exhibit a mixed torus and apply the Christian–Menke splitting argument, from which it follows that no symplectic rational homology ball filling can exist.   To this end, we note that if we remove one of the legs from the surgery diagram, we will obtain a lens space (that is, a Seifert fibered space with two singular fibers), and adding the leg back into the diagram will be doing surgery on some torus knot in the lens space. Specifically, in Figure~\ref{fig:s1s2} if we replace the labels $r$ and $-r$ with $-1/s_i$ and $-1/s_j$, for some $i,j\in\{1,2,3\}$, we will obtain the lens space, and $F$ will represent the torus knot that must be surgered to recover our original Seifert fibered space. To be specific, we can assume that we removed the third leg, so our lens space is given by $-1/s_1$ and $-1/s_2$ surgeries in the figure. 

We now consider the contact structure on this lens space. It will be given by a minimal path in the Farey graph that starts at $-1/s_1$ and moves clockwise to $1/s_2$ (note the flip in the sign from $-1/s_2$ to $1/s_2$, this is for the same reason we have a flip in sign in the proof of Lemma~\ref{sfsons1s2}) with signs on all but the first and last edge. We claim that all the signs on the edges from $-1/s_1$ to $0$ are all the same and opposite from those going from $0$ to $1/s_2$. (We note that in the case when $s_1=n>1$ there is only one edge from $-1/s_1$ to $0$ we will deal with this case below and so for now assume that this is not the case.) To see why this claim is true, note that when thinking of these tori as coming from surgery on a Legendrian link in $S^1\times S^2$ they both have lower meridians, but when thinking about them as describing the lens space, one has a lower meridian and the other has an upper meridian. When switching between upper and lower meridians, the co-orientation on the boundary of the torus changes. 
Thus, since both legs were stabilized the same way, the solid tori corresponding to those legs have basic slices of the same sign. However, when gluing the two solid tori together, we consider one of them to have an upper meridian and hence the signs of the basic slices in the solid torus change sign. 
Thus, we see a mixed torus in the middle of our lens space and the third leg is obtained by doing surgery on a torus knot sitting on this torus. 

The knot $F$ on which we must {perform surgery}  to get our original Seifert fibered space back is a Legendrian divide on the convex torus with dividing slope $0$. We note that since $s_3\in(0,1)$ we know $-1/s_3<-1$ and so we must stabilize $F$ at least once to perform the correct surgery. If we consider a ruling curve $F'$ on the convex torus for slope $1$ in our lens space (note from our description of the lens space above this exists), then it will be a stabilization of $F$.  The sign of the stabilization is determined by the sign on the edge from $0$ to $1$, and one can check that it will be opposite to that sign. So this is the stabilization of $F$ that needs to be done to perform the contact surgery to obtain our original contact structure. Thus in our original contact structure, we have a mixed torus of slope $0$ with a basic slice of slope $-1/n$ and $0$ on one side and $0$ and $1$ on the other side. Thus the possible $e$-splittings along this torus can have meridional slopes $\{\infty, -1/k\}$ where $k=1,\ldots n-1$. The $e$-splitting will excise the first leg and replace it with a solid torus with meridional slope $e$, and also cap off the solid torus given by the first leg with another solid torus with meridional slope $e$. If $e=1/k$ for $k>1$, then the first manifold will be a small Seifert fibered space with $e_0\geq 0$ again, and the second manifold will be a lens space. As argued above, we cannot attach a round $1$-handle to any fillings of these spaces to obtain a rational homology ball. If $e=-1$ or $\infty$ then both manifolds coming from the splitting will be a lens space, but if one of these is not $S^1\times S^2$ filled by $S^1\times D^3$ then, as noted above, the original contact manifold cannot have a rational homology ball filling. Only the first manifold in the splitting could be $S^1\times S^2$, and in this case, since the signs of the stabilizations are the same, we would get the overtwisted contact structure on $S^1\times S^2$ by Lemma~\ref{lem: framing}, and hence it cannot be symplectically fillable. Thus there cannot be a symplectic rational homology ball filling of the original contact manifold. 


We are left to consider the case when $s_1=n>1$. Recall we assumed this so that the path from $-1/s_1$ to $0$ describing a contact structure on a solid torus would have a sign on the edge adjacent to $0$. When $-1/s_1=-1/n$ this will not be the case. But then the path from $-1/s_1$ to $0$ is a single edge with no sign and corresponds to the unique tight contact structure on a solid torus with longitudinal dividing curves. Thus this solid torus is a neighborhood of a Legendrian curve. We can stabilize this Legendrian curve with either sign and the complement of this stabilized curve will be described by a path in the Farey graph from $-1/(n+1)$ to $0$ with a sign depending on which stabilization is done. Thus we can still arrange for a mixed torus with dividing slope $0$ and carry out the argument as above. 
\end{proof}

We are now ready to prove Corollary~\ref{possfillse0pos}. Recall that this corollary states the following:  when $e_0\geq 0$, there are at most three contact structures (up to isomorphism) on $Y(e_0;r_1,r_2,r_3)$ that can admit symplectic rational homology ball fillings. 

\begin{proof}[Corollary~\ref{possfillse0pos}]
The proof is almost identical to the proof of Corollary~\ref{possiblefillingsfore0m2}, except the fact that the consistent contact structure cannot have a rational homology ball filling in this case by Theorem~\ref{main3}. 
\end{proof}

\section{Seifert fibered spaces with $e_0=-1$}\label{Lnegm1}
In this section we  establish an obstruction for a small Seifert fibered space with $e_0=-1$ to carry a contact structure that admits a symplectic rational homology ball filling and  prove results restricting certain Brieskorn spheres from admitting such fillings. The latter are obtained by analyzing Dehn surgeries on torus knots.

\subsection{Restricting rational homology ball fillings}
In this subsection, we prove Theorem~\ref{Lspace} that says 
the Seifert fibered space $Y(-1; r_1, r_2, r_3)$ where $1>r_1\geq r_2\geq r_3>0$ admits no 
 tight contact structure that is symplectically filled by a rational homology ball whenever $r_1+r_2+r_3>1>r_1+r_2$.

\begin{proof}[Proof of Theorem~\ref{Lspace}]
The proof we present here is modeled on a result of Lecuona-Lisca \cite[Theorem~$1.4.$]{LL2011}, who prove a stronger result under stronger assumptions. 

Suppose that $Y$ is filled by a symplectic rational homology ball, say $X$, which of course has $b_2^+(X)=0$. Now consider the orientation reversal manifold $-Y(-1; r_1, r_2, r_3)=Y(-2, s_1, s_2, s_3)$, where a simply Kirby calculus argument yields that $s_i=1-r_{4-i}$, $i=1, 2, 3$. Note that the Seifert invariants $s_i$ satisfy $1>s_1\geq s_2\geq s_3$ as $1>r_1\geq r_2\geq r_3$. Now the assumption $r_1+r_2+r_3>1$ immediately implies that the star-shaped plumbing diagram in Figure~\ref{fig:plumbing} for $Y(-2, s_1, s_2, s_3)$ has negative definite intersection form. Let $X'$ be the plumbed $4$-manifold given in the diagram. With this in place, we form the smooth closed 4-manifold $Z=X\cup X'$ which has a negative definite intersection form. In particular, by Donaldson's theorem \cite{Donaldson87}, $Z$ has a diagonalizable intersection form. On the other hand, the assumption $r_1+r_2<1$ says $s_2+s_3=(1-r_2)+(1-r_1)>1$ which is enough to conclude, by \cite[Lemma~$3.3$]{LL2011}, that the intersection form of $X$ does not embed in a standard diagonal lattice. Thus, $Y$ does not admit a tight structure that is filled by a symplectic rational homology ball. 
\end{proof}

Under the further assumption of $Y$ being an L-space, Lecuona-Lisca \cite[Theorem~1.4]{LL2011} prove that $Y$ has no symplectic fillings at all. The key is that any such symplectic filling will be negative definite by \cite[Theorem~1.4]{OzsvathSzabo04a}.   

\subsection{Surgeries on torus knots}\label{surgeryontorus}
For a small Seifert fibered space $Y=Y(-1; r_1, r_2, r_3)$,  being an $L$-space plays an important role in the classification of tight contact structures on $Y$.  So we briefly review how one can determine whether $Y$ is an L-space. It follows from \cite[Theorem~1.4.]{LiscaStipsicz07} that a small Seifert fibered space $Y=Y(-1; r_1, r_2, r_3)$ is an L-space if $r_1+r_2+r_3\geq \frac{3}{2}$ or $r_1+r_2\geq 1$, and it is not an L-space if $r_1+r_2+r_3<1$. When $1<r_1+r_2+r_3<\frac{3}{2}$ either could happen and in general there is no closed formula that can determine if $Y$ is an L-space or not.

We are not aware of any non-$L$-space small Seifert fibered space, regardless of its $e_0$ value, with a symplectic rational homology ball filling. It is possible that there are no such examples. Our goal here is to provide some evidence for this through Theorem~\ref{nonLspace}. To this end, we consider the $3$-manifold $S^3_{T_{p,q}}(r)$ --- the result of smooth $r$-surgery on the positive torus knot $T_{p,q}$ for any rational number $r<0$, and where we adopt the convention that $p, q$ are relatively prime integers with $0<p<q$. This $3$-manifold is the small Seifert fibered space $Y(-1; \frac{p-q^*}{p}, \frac{q-p^*}{q}, \frac{1}{pq-r})$ where $p^*$ and $q^*$ are multiplicative inverses of $p$ and $q$, modulo $q$ and $p$, respectively. Clearly,  the sum of the Seifert invariants here is less than one, and hence this manifold is not an L-space. In Theorem~\ref{nonLspace}, we will show that for any $r$, an integer less than $-1$ if $\{p,q\}=\{2,3\}$, and otherwise a negative integer or a rational number of the form $-\frac{1}{n}$ for some positive integer $n$, the $3$-manifold $S^3_{T_{p,q}}(r)$ does not bound a symplectic rational homology ball. 

As an immediate corollary,  we obtain Theorem~\ref{brieskorn}, which says that for $(p,q)\neq (2,3)$, no Brieskorn homology sphere of the form $\Sigma(p,q, pqn+1)$  bounds a symplectic rational homology ball. 

\begin{proof}[Proof of Theorem~\ref{brieskorn}]
When $r= -\frac{1}{n}$ for some integer $n$, the manifold $S^3_{T_{p,q}}(r)$ will be the Brieskorn homology sphere $\Sigma(p,q,pqn+1)$, thus Theorem~\ref{nonLspace} rules out symplectic rational homology ball fillings of any contact structure on $\Sigma(p,q,pqn+1)$ for $(p,q)\neq (2,3)$.
\end{proof}

To prove Theorem~\ref{nonLspace} we note a classification result for surgeries on torus knots (see also \cite[Theorem~$1.9$]{MarkTosunT18} for when $r=-\frac{1}{n}$ and $(p,q)=(2,3)$).

\begin{theorem}[Etnyre-Min-Tosun-Varvarezos 2024, \cite{EtnyreMinTosunPre}\footnote{More general results will appear in the forthcoming \cite{EtnyreMinTosunPre}, but this result follows from standard techniques in the field, details of which can be found in \cite{EMTVnote} .}]\label{EMT}
Let $r$ be any negative rational number with continued fraction expansion as $r=[-a_0, -a_1, \cdots, -a_k]$ where $a_0\geq 1, a_i\geq 2,$ for $i>0$. Then, up to isotopy, there exactly $(pq-p-q+a_0)(a_1-1)\cdots(a_k-1)$ symplectically fillable contact structures on $S^3_{T_{p,q}}(r)$. 
\end{theorem}
        
We now make some preparations for proving Theorem~\ref{nonLspace}. Initially, we will keep the discussion, in terms of $p,q$ and $r$ values, as general as possible. 

By Theorem~\ref{EMT} we know the exact number of fillable contact structures on $S^3_{T_{p,q}}(r)$. 
Negative contact $(-pq+p+q+r)$-surgery on the Legendrian positive torus knot $L_{p,q}$ with $tb(L_{p,q})=pq-p-q$, will result in the smooth manifold  $S^3_{T_{p,q}}(r)$ with a contact structure depending on the particular contact surgery performed (recall that the contact surgery is determined by a choice of tight contact structure on the surgery torus). There is a standard procedure (see \cite[Page 5]{DingGeigesStipsicz04}) to convert this contact negative rational surgery to a sequence of Legendrian surgeries on an appropriately stabilized $L_{p,q}$ and its push-offs which might include further stabilizations.

However, for our arguments below it will be more convenient to use a slightly different but equivalent description of this Legendrian surgery. Namely, using a sequence of (reverse) slam dunk operations from Figure~\ref{fig:sd} we can convert smooth $r$-surgery on $T_{p,q}$ into surgery on a chain of knots as in Figure~\ref{EMT1} where $K=K_0$ is the torus knot $T_{p,q}$ and $K_i$ are the unknots for $i=1, \cdots, k$, and $r=[-a_0,-a_1,\ldots, -a_k]$.
\begin{figure}[htb]{
\begin{overpic}
{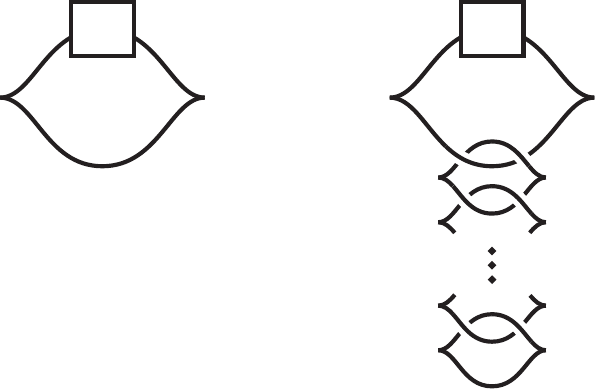}
\put(44, 170){$K$}
\put(232, 170){$K$}
\put(6,92){$(-pq+p+q+r)$}
\put(112, 155){$(-pq+p+q-a_0)$}
\put(160, 99){$(-a_1+1)$}
\put(160, 15){$(-a_k+1)$}
\end{overpic}}
\caption{Contact surgery diagram for smooth $r$-surgery on the maximal Thurston-Bennequin invariant positive $(p,q)$-torus knot is shown on the left, and equivalent Legendrian surgery diagram is shown on the right.}
\label{EMT1}
\end{figure}
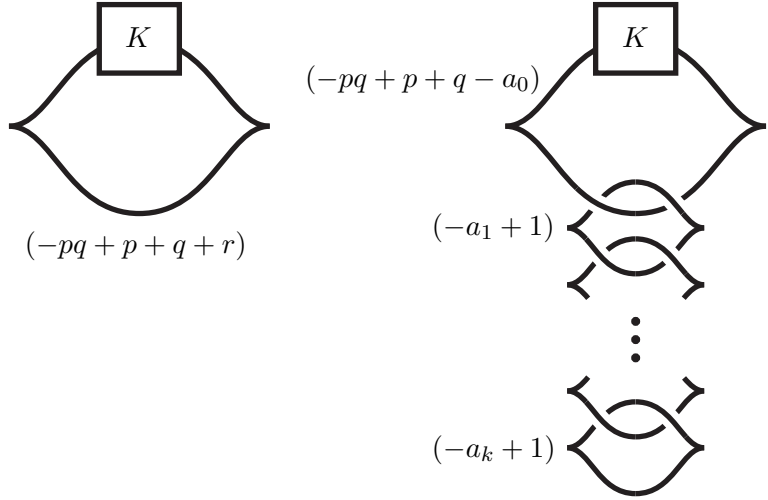
With this, contact  $(-pq+p+q+r)$-surgery along the original Legendrian knot $L_{p,q}$ can be described as Legendrian surgery along the link $L_0\cup L_1\cup \cdots \cup L_k$ where

\begin{itemize}
\item $L_0$ is the Legendrian knot represented by $L_{p,q}$ with $pq-p-q+a_1-1$ stabilizations, 

\item for $i\geq 1$, $L_{i}$ is the Legendrian unknot with $tb(L_i)=-1$ and $a_i-2$ stabilizations. 
\end{itemize} 

Let $(X_{p,q,r}, J_{p,q,r})$ denote the Stein 2-handlebody described by the Legendrian surgery diagram in Figure~\ref{EMT1}. 
The possible choices of stabilizations mentioned in the above construction give $(pq-p-q+a_0)(a_1-1)\cdots(a_k-1)$ Stein fillable and pairwise distinct contact structures, \cite{LiscaMatic97}. This is precisely the number of contact structures predicted in Theorem~\ref{EMT} and thus all contact structures on $S^3_{T_{p,q}}(r)$ for $r<0$ can be achieved via Legendrian surgery on the appropriate Legendrian realization of Figure~\ref{EMT1}. We call a contact structure \dfn{consistent} if all the stabilizations of the Legendrian link have the same sign. Otherwise we call it \dfn{inconsistent}. 
We next state our first preliminary result.   

\begin{proposition}\label{rednumb}
If $\xi$ is an inconsistent contact structure on $S^3_{T_{p,q}}(r)$, then it is not symplectically filled by a rational homology ball.  
\end{proposition}
  
\begin{proof}
The proof follows the outline of the proof of Theorem~\ref{main1}. If $\xi$ is an inconsistent contact structure then there is a mixed torus in $(S^3_{T_{p,q}}(r),\xi)$.  We can use Theorem~\ref{splittingthm} to split any symplectic filling $X$ of $(S^3_{T_{p,q}}(r),\xi)$ into a symplectic filling of $(S^3_{T_{p,q}}(r'),\xi)$ union a lens space, for some $r'<0$, and $X$ can be recovered by attaching a round $1$-handle. But since $(S^3_{T_{p,q}}(r'),\xi)$ cannot be filled by a rational homology $S^1$ (see the proof of Theorem~\ref{main3}) and neither can a lens space we see from Remark~\ref{homologyofsplitting} that $X$ could not have been a rational homology ball. 
\end{proof}

Thus when proving Theorem~\ref{nonLspace} we only need to focus on the consistent contact structure $\xi^{p,q}_{can}$.

\begin{proof}[Proof of Theorem~\ref{nonLspace}] 
We will show that $\theta(\xi^{p,q}_{can}) \neq -2$. Thus by Lemma~\ref{lem: rhb}, $\xi^{p,q}_{can}$ cannot be symplectically filled by a rational homology ball.

We first gather certain smooth and symplectic data of the Stein manifold $X_{p,q,r}$ described in Figure~\ref{EMT1}.
It is easy to see $\chi(X_{p,q,r})=k+2$ and $\sigma(X_{p,q,r})=-k-1$ where recall $k+1$ is the number of the entries in the continued fraction for $r$. The Euler characteristic calculation is straightforward from the diagram. As for proving $X_{p,q,r}$ is negative definite, we recall that $S^3_{T_{p,q}}(r)=Y(-1; \frac{p-q^*}{p}, \frac{q-p^*}{q}, \frac{1}{pq-r})$, and so we have 
\[
e(S^3_{T_{p,q}}(r)) =1-\left(\frac{p-q^*}{p}+ \frac{q-p^*}{q}+\frac{1}{pq-r}\right)>0.
\] 
where $e$ denotes the rational Euler number. It follows that $b^+_2(X_{p,q,r})=0$, by \cite[Theorem~5.2]{neumannraymond}. 

We now calculate $c_1^2(X_{p,q,r}, J_{p,q,r})$. For this we first note that the intersection matrix $I_{X_{p,q, r}}$  of $X_{p,q,r}$ is   

$$I_{X_{p,q, r}}=\left[\begin{array}{cccccc}
-a_0 & 1 & & & & \\
1 & -a_1 & 1 & &  &  \\
& 1 &  & & & \\
& &  & \ddots & 1 \\
& & &  1& -a_{k-1} &1 \\
&  & & & 1 & -a_k
\end{array}\right] $$

 We also need the rotation vector: ${\bf r} _{p,q}=(r_1(L_0), \cdots ,r_k(L_k))^T$ where $r_i(L_i)$ is the rotation number of the Legendrian knot $L_i$. The rotation vector for $\xi^{p,q}_{can}$ can be  given as ${\bf r}_{can}=(pq-p-q+a_0-1, a_1-2, \cdots, a_k-2)^T$ and
\begin{equation}\label{hdata}
c_1^2(X_{p,q,r}, J^{can}_{p,q,r})={\bf r}_{can}^T I_{X_{p,q, r}}^{-1} {\bf }{\bf r}_{can}.
\end{equation}

We complete the proof in two cases.
\smallskip

\noindent
{\bf Case 1:} (Smooth $-\frac{1}{n}$-surgery on $T_{p,q}$ for $n\geq 1$.)  Note that $-\frac{1}{n}=[-1, -2, \cdots, -2]$ where the length of the continued fraction block is $n$ (and to be consistent with the initial notation above this says $k+1=n$).  We implement the formula in Equation~\ref{hdata}. Note that ${\bf r}_{can}=(pq-p-q, 0, \cdots, 0)^T$. Using this we quickly solve the linear system $I_{X_{p,q, r}} {\bf x}={\bf r}_{can}$ for ${\bf x}=[x_1, \cdots, x_n]$ and notice that $c_1^2(X_{p,q,r}, J^{can}_{p,q,r})$ is computed  as the dot product of ${\bf x}$ and ${\bf r}_{can}$.
We find that $x_1=-n(pq-p-q), x_2=-(n-1)(pq-p-q),$ and the rest of the $x_i$ are easily computed from this. So $c_1^2(X_{p,q,r}, J^{can}_{p,q,r})=-n(pq-p-q)^2$, and as above $\chi(X_{p,q,r})=n+1$ and $\sigma(X_{p,q,r})=-n$. Thus,  
\begin{align*}
\theta(\xi^{p,q}_{can})&=c_1^2(X_{p,q,r}, J^{can}_{p,q,r})-2\chi(X_{p,q,r})-3\sigma(X_{p,q,r})\\ &=-n(pq-p-q)^2+n-2,
\end{align*}
and we have \[\theta(\xi^{p,q}_{can})=-2 \text{ if and only if } (p,q)=(2,3).\]

When $(p, q)\neq (2,3)$, then $\theta(\xi^{p,q}_{can})<-2$. So, it never arises as the boundary of a symplectic rational homology ball. 
\begin{remark}
It is interesting to note that if we consider the contact structure $\xi'$ obtained from stabilizing $L_1$ positively $\frac{pq-p-q+1}{2}$ times and negatively $\frac{pq-p-q-1}{2}$ times, then we get that $c_1^2=-n$ and $\theta(\xi')=-2$. Of course, $\xi'$ is inconsistent, and hence Proposition~\ref{rednumb} rules out $\xi'$ to be symplectically filled by a rational homology ball.  Thus we see that to rule out rational homology ball fillings we need to use both the technique of splitting fillings along mixed tori and computations of the $\theta$-invariant. 
\end{remark}

\noindent
{\bf Case 2:} (Smooth $r=-n$-surgery on $T_{p,q}$ for $n\geq 1$.) There is only one component in the Legendrian surgery diagram which has $r_1(K)=pq-p-q+n-1$, and so $c_1^2(X_{p,q,r}, J^{can}_{p,q,r})=-\frac{(pq-p-q+n-1)^2}{n}$. We also have $\chi(X_{p,q,r})=2$ and $\sigma(X_{p,q,r})=-1$. Thus, 
\[
\theta(\xi^{p,q}_{can})=-\frac{(pq-p-q+n-1)^2}{n}-1.
\] 
Similar to above we note that $\theta(\xi^{p,q}_{can})=-2$ if and only if $(p,q)=(2,3)$ and $n=1$. This completes the proof.
\end{proof}

\section{Spherical $3$-manifolds}\label{sphericalsec}
In this section we will determine precisely which spherical $3$-manifolds can bound symplectic rational homology balls. A theorem of Choe and Park \cite{ChoePark2021} says that a spherical $3$-manifold $Y$ bounds a smooth rational homology ball if and only if $Y$ or $-Y$ is homeomorphic to one of the following manifolds: 
\begin{enumerate}
\item $L(p, q)$ such that $p/q \in\mathcal{R}$,
\item\label{2} $D(p, q)$ such that $(p-q)/q' \in\mathcal{R}$,
\item\label{3} $T_3$, $T_{27}$ and $I_{49}$, 
\end{enumerate}
where $p$ and $q$ are relatively prime integers such that $0 < q < p$, and $0 < q' < p - q$ is
the reduction of $q$ modulo $p - q$. So we will only consider these $3$-manifolds when looking for spherical $3$-manifolds bounding symplectic rational homology balls. 

\begin{remark} \label{rem: cond} The condition $(p-q)/q' \in\mathcal{R}$  in Item~\eqref{2} comes from Lecuona's work \cite{Lecuona2019} combined with the work of Lisca \cite{Lisca2007} as follows. As a corollary to \cite[Proposition 3.1]{Lecuona2019}, the spherical $3$-manifold $D(p,q)$ is rational homology cobordant to the lens space $L(p-q, q)$ and according to \cite[Theorem 1.2]{Lisca2007},  $L(p-q, q)$ bounds a rational homology ball if and only if $(p-q)/q \in \mathcal{R}$.  The caveat is that if $a_0 =2$ in the continued fraction expansion $p/q=[a_0, a_1, \ldots, a_k]$, then $p-q < q$, but  since $L(p-q,q) \cong L(p-q,q')$ for $q' \equiv q$ modulo $p-q$, we deduce that $D(p, q)$ bounds a rational homology ball if and only if $(p-q)/q' \in\mathcal{R}$, where $0 < q' < p - q$ is the reduction of $q$ modulo $p - q$. If $a_0 \geq 3$, however,  then the condition in Item~\eqref{2} is equivalent to $(p-q)/q\in\mathcal{R}$.   \end{remark}

The ${\bf D}$-type spherical $3$-manifolds appearing in Item~\eqref{2} above are also known as prism manifolds and they are defined as follows. Suppose that  $(p,q)$ is a coprime pair of integers such that $0 < q < p$ and let $[a_0,a_1,\ldots, a_k]$ be the continued fraction expansion of $p/q$ where the $a_i\geq 2$ for all $i$. 

The spherical $3$-manifold $D(p, q)$  is defined by the plumbing diagram in Figure~\ref{fig: type2}, where we assume that $k \geq 1$, since otherwise it is a lens space.  
\begin{figure}[htb]{
\begin{overpic}
{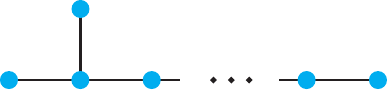}
\put(47, 44){$-2$}
\put(-6,-3){$-2$}
\put(30,-3){$-a_0$}
\put(65,-3){$-a_1$}
\put(138,-3){$-a_{k-1}$}
\put(174,-3){$-a_k$}
\end{overpic}}
\caption{The plumbing diagram for $D(p,q)$, where $p/q = [a_0, a_1, \ldots, a_k]$.}
\label{fig: type2}
\end{figure}
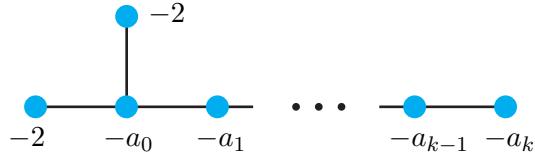

The spherical $3$-manifolds $T_3$, $T_{27}$, and $I_{49}$, listed in Item~\eqref{3} above, are described,  respectively, by the plumbing diagrams depicted from left to right in Figure~\ref{fig: type3}.

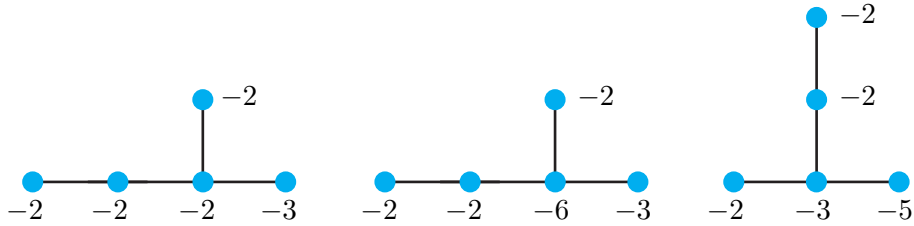
\begin{figure}[htb]{
\begin{overpic}
{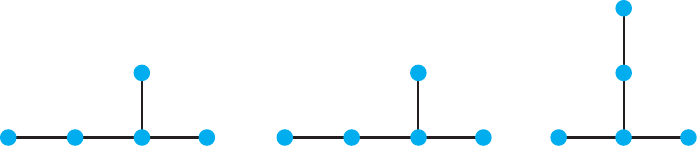}
\put(75, 40){$-2$}
\put(-6,-3){$-2$}
\put(26,-3){$-2$}
\put(59,-3){$-2$}
\put(90,-3){$-3$}
\put(210, 40){$-2$}
\put(128,-3){$-2$}
\put(160,-3){$-2$}
\put(193,-3){$-6$}
\put(224,-3){$-3$}
\put(309, 40){$-2$}
\put(309, 71){$-2$}
\put(259,-3){$-2$}
\put(292,-3){$-3$}
\put(323,-3){$-5$}
\end{overpic}}
\caption{The plumbing diagrams for $T_3$, $T_{27}$, and $I_{49}$, respectively.}
\label{fig: type3}
\end{figure}

We will analyze each type of spherical $3$-manifolds listed above to see if it can admit symplectic rational homology ball fillings. 

\subsection{The case of $L(p,q)$} 

This case has already been resolved completely as we elaborated in the introduction. Nevertheless, we analyze this case from another point of view.  Suppose that  $(p,q)$ is a coprime pair of integers such that $0 < q < p$.   Let $p/q=[a_0, a_1, \ldots, a_k]$ be the continued fraction expansion, where $a_i \geq 2$ for all $0 \leq i \leq k$. We begin by recalling a convenient definition from \cite{Lisca2007}. 

\begin{definition} \label{def: i} With $p,q$ and $a_i$ as above, we set $$I(p/q)=\sum_{i=0}^{k} (a_i-3).$$ 
\end{definition}

We now compute the $\theta$-invariant of the universally tight contact structure on $L(p,q)$, which we denote by $\xi_{can}$. 
\begin{proposition} \label{prop: thetalens} For the canonical contact structure $\xi_{can}$ on $L(p,q)$,   we have 
\begin{equation}\label{eq: thetalens}
 \theta (\xi_{can}) = -\left(I(p/q) + \dfrac{2+q+q^*}{p} \right), 
\end{equation}
where $0 < q^* < p$ is the multiplicative inverse of $q$ mod $p$.  
\end{proposition}

We will give a proof of Proposition~\ref{prop: thetalens} in the Subsection~\ref{subsec: thetalens}. A useful observation about the $\theta$-invariant of the canonical contact structure on a lens space follows from Proposition~\ref{prop: thetalens}. Recall that $\mathcal{O}$ in the lemma below is defined by Equation~\ref{eq: O}.

\begin{lemma} \label{lem: lpq} The canonical contact structure $\xi_{can}$ on the lens space $L(p,q)$ satisfies
$$-2 \leq \theta (\xi_{can}),$$ provided that  $p/q \in\mathcal{R}$, with equality only when $p/q \in \mathcal{O} \subset \mathcal{R}$. 
\end{lemma}

\begin{proof} According to Lisca \cite{Lisca2007}, the condition $p/q \in\mathcal{R}$ is equivalent to the condition that $L(p, q)$ bounds a rational homology ball.  By Equation~\ref{eq: O}, $p/q \in \mathcal{O}$ means that $p=m^2$ and $q=mh-1$ for some coprime integers $0< h < m$.  For $p/q \in \mathcal{O}$, we know that $\theta (\xi_{can})=-2$, since the contact $3$-manifold $(L(p,q), \xi_{can})$ is Stein fillable in this case. One can also verify this fact using Formula~\ref{eq: thetalens} as follows. 

A final $2$-expansion of a continued fraction $[c_1, \ldots, c_k]$  is defined to be either $$[2, c_1, \ldots, c_{k-1},1+c_k]\;\; \mbox{or}\;\;  [1+c_1, c_2, \ldots, c_k, 2].$$ It is well-known (see, for example, \cite[Proposition 4.1]{SSW08}) that every rational number that belongs to the set $\mathcal{O}$ can be described by a continued fraction obtained by starting from $[4]$ and applying finitely many final $2$-expansions.  It follows that $I(p/q)=1$ for any $p/q \in \mathcal{O}$, since $I  (4/1)$=1 and a final $2$-expansion does not change the $I$-value. Moreover, for  $p=m^2$ and $q=mh-1$,  we have $q^* = m (m-h)-1$ since $q q^* \equiv 1 \mod p$ and thus  $$2+q+q^*=2+mh-1+m (m-h)-1=m^2=p,$$ implying  by Formula~\ref{eq: thetalens} that if $p/q \in \mathcal{O}$, then $\theta (\xi_{can})=-2$. 

For $p/q \in \mathcal{R}\setminus \mathcal{O}$, the inequality $I(p/q) \leq 0$ holds by Lisca's work \cite{Lisca2007} (see also, \cite[Appendix~A.2]{AcetoGollaLarsonLecuona2020pre}) and thus assuming further that $p \geq q+2$, we have $-2 < \theta (\xi_{can}),$ by Formula~\ref{eq: thetalens},  since  $q^*<p$. When $p=q+1$, the quotient $p/q$ belongs to  $\mathcal{R}$ if and only if $q=3$ and in that case, one can verify that $\theta (\xi_{can})=1$.  
\end{proof} 
\begin{remark} The converse of Lemma~\ref{lem: lpq} does not hold. For example, we have $-2 < -1 = \theta (\xi_{can})$ for $L(8,3)$,  but $8/3 \notin \mathcal{R}$.  
\end{remark} 

We can now determine which contact structures on lens spaces admit symplectic rational homology ball fillings. 
\begin{proposition}\label{prop1}
A contact structure $\xi$ on a lens space $L(p,q)$ admits a symplectic rational homology ball filling if and only if $p/q\in\mathcal{O}$ and $\xi=\pm \xi_{can}$. 
\end{proposition}
\begin{proof}
By Proposition~\ref{canminimizes} (more precisely the remark after the proposition) and Lemma~\ref{lem: lpq} we see that $\theta(\xi)\geq-2$ for all tight contact structures $\xi$ on $L(p,q)$ and $\theta(\xi)=-2$ if and only if $p/q\in\mathcal{O}$ and $\xi=\pm\xi_{can}$. Thus by Lemma~\ref{lem: rhb}  $\xi$ can bound a symplectic rational homology ball if and only if $p/q\in\mathcal{O}$ and $\xi= \pm \xi_{can}$; moreover, in this case it actually does bound a symplectic rational homology ball by \cite{Lisca08} (see also \cite[Figure~2]{EtnyreRoy21}).
\end{proof}

\subsection{The case of $D(p,q)$} 
Suppose that  $(p,q)$ is a coprime pair of integers such that $0 < q < p$ and let
$p/q=[a_0, a_1, \ldots, a_k]$ be the continued fraction expansion, where $a_i \geq 2$ for all $0 \leq i \leq k$. Let $X(p,q)$ be the smooth $4$-manifold described by the plumbing graph in Figure~\ref{fig: type2}.

Note that $X(p,q)$ is obtained by attaching $(k+3)$ $2$-handles to the $4$-ball along the unknots $K_{-1}, K_0, K_1, \ldots,  K_{k+1}$ with framings  $-2$, $-a_0$, $-a_1$, $\ldots$, $-a_k$, $-2$, respectively, as  indicated in  Figure~\ref{fig: type2}. The boundary of  $X(p,q)$ is the spherical $3$-manifold $D(p,q)$. Since $D(p,q)$ is a small Seifert fibered space with $e_0\leq -2$, and when $e_0=-2$ it is an $L$-space, our discussion in the introduction shows that any tight contact structure on $D(p,q)$ is obtained by the appropriate Legendrian realization of the $K_i$.

Let $\xi$ be any given tight (hence Stein fillable) contact structure on $D(p,q)$, which is represented by a Legendrian surgery diagram as described above.  By fixing a Legendrian surgery diagram for the contact structure $\xi$, we also fix a Stein surface $(X(p,q), J_\xi)$ inducing  $\xi$ on its boundary $D(p,q)$. 

\begin{lemma} \label{lem: theta} For the canonical contact structure $\xi_{can}$ on $D(p,q)$,   we have 
\begin{equation}\label{eq: c2}
c_1^2 (X(p,q), J_{\xi_{can}})  = 2k+3-(a_0 +a_1+\cdots+a_k) - \dfrac{1}{[a_k, a_{k-1}, \ldots, a_1, a_0-1]}.
\end{equation}
\end{lemma}

\begin{proof}  The intersection matrix of the $4$-manifold $X(p,q)$ can be given as $$I_k=\left[\begin{array}{cccccc}
-2 & 1 & & & & \\
1 & -a_0 & 1 & &  & 1 \\
& 1 & -a_1 & \ddots & & \\
& & \ddots & \ddots & 1 \\
& & & 1 & -a_k & \\
& 1 & & & & -2
\end{array}\right] $$ where the top-left $-2$ in the matrix corresponds to the left-most $-2$-framed unknot in Figure~\ref{fig: type2}, and all blank entries are assumed to be zero.  By a slight abuse of notation, we denote the intersection matrix above by  $I_k$, and set    
\begin{equation} \label{eq: B} B_0=1, \;\;\; B_1=-(a_0-1), \;\;\; \mbox{and}\;\;\; B_{r+1}=-(a_rB_r+B_{r-1})\;\;\; \mbox{for all} \;\;\; 1 \leq r \leq k. 
\end{equation}
We claim that $\det I_k = 4 B_{k+1}.$ To prove our claim we use induction on $k$. For $k=0$, we have 
$$\begin{vmatrix}
-2 & 1 & 0 \\
1 & -a_0 & 1 \\
 0 & 1 & -2
\end{vmatrix} = -2 \begin{vmatrix} 
-a_0 & 1 \\
 1 & -2
\end{vmatrix} - \begin{vmatrix} 
1 & 1 \\
 0 & -2
\end{vmatrix} = -4a_0 + 4 = 4 B_1, $$ where the last equality follows from Definition~(\ref{eq: B}). 
For $k=1$, by expanding along the third row, and using the result for $k=0$, we have 
$$\underbrace{\begin{vmatrix}
-2 & 1 & 0 & 0\\
1 & -a_0 & 1 & 1\\
 0 & 1 & -a_1 &  0 \\
 0 & 1 & 0 &   -2 
\end{vmatrix}}_{\mbox{expand along 3rd row}}= - \underbrace{\begin{vmatrix}
-2 & 0 & 0 \\
1 & 1 & 1 \\
 0 & 0 & -2
\end{vmatrix}}_{4B_0}  - a_1  \underbrace{\begin{vmatrix}
-2 & 1 & 0 \\
1 & -a_0 & 1 \\
 0 & 1 & -2
\end{vmatrix}}_{4B_1}  =  4 (-a_1B_1 - B_0)=4B_2,$$ where the last equality follows from Definition~(\ref{eq: B}). 
Similarly, by expanding along the row just above the last one, and using induction, we have 
$$\det I_{k+1} = \begin{vmatrix}
-2 & 1 \\
1 & -a_0 & 1 & &  & & 1 \\
& 1 & -a_1 & \ddots & & \\
& & \ddots & \ddots & 1 \\
& & & 1 & -a_{k-1} & \\
& & &  & 1& 1& \\
& 1 & & & & & -2
\end{vmatrix}  - a_{k+1} \underbrace{\begin{vmatrix}
-2 & 1 \\
1 & -a_0 & 1 & &  &  1 \\
& 1 & -a_1 & \ddots & & \\
& & \ddots & \ddots & 1 \\
& & & 1 & -a_{k} & \\
 & 1 & & & &  -2
\end{vmatrix}}_{4B_{k+1}} $$ and expanding the first determinant again along the row just above the last one, and using induction, we have 
$\det I_{k+1} = 4(-B_k -  a_{k+1} B_{k+1}) =  4 B_{k+2}$, by Definition~(\ref{eq: B}), proving our claim. 

Let ${\bf r}_{can}$ denote the rotation vector for the canonical contact structure $\xi_{can}$, obtained by Legendrian realizing each $K_i$ with maximal possible rotation number, i.e., $$ {\bf r}_{can} = [0 \;\; a_0-2 \;\; a_1-2 \; \cdots \; \;a_k-2\;\; 0]^T.$$  Note that 
\[
c_1^2 (X(p,q), J_{\xi_{can}})= {\bf r}_{can}^T I^{-1}_{k} {\bf r}_{can},
\]
by \cite{Gompf98} and we compute this quantity in three steps: 
\begin{subequations}\label{eqn:binomi}
      \begin{align}
       {\bf r}_{can}^T I^{-1}_{k} {\bf r}_{can} & =\dfrac{1}{B_{k+1}} \sum_{i=0}^{k} (a_i -2) (-B_{k+1} + (-1)^{k+1-i} B_i)\\
 & = 2k+3 -(a_0 +a_1+\cdots+a_k) + \dfrac{B_k}{B_{k+1}} \\
 & = 2k+3-(a_0 +a_1+\cdots+a_k) - \dfrac{1}{[a_k, a_{k-1}, \ldots, a_1, a_0-1]} 
      \end{align}
    \end{subequations}
where the subequations in ~(\ref{eqn:binomi}) can be proved by induction on $k$ as follows.

To prove the subequation~(\ref{eqn:binomi}c), we simply observe that 
    $$- \dfrac{1}{[a_k, a_{k-1}, \ldots, a_1, a_0-1]} = - \dfrac{1}{{a_k - \dfrac{1}{[a_{k-1}, \ldots, a_1, a_0-1]}}}= - \dfrac{1}{{a_k + \dfrac{B_{k-1}}{B_k}}}= \dfrac{B_k}{B_{k+1}},$$ where the last equality follows from Definition~(\ref{eq: B}). 
    
To prove the subequation~(\ref{eqn:binomi}b), we first observe that 
    $$ \sum_{i=0}^{k} (-1)^{k+1-i}  (a_i -2)  B_i = - (a_k-2) B_k - \sum_{i=0}^{k-1} (-1)^{k-i}  (a_i -2) B_i $$ and by induction we assume that the quantity under the summation sign on the right-hand side is equal to $B_k+ B_{k-1}$, implying that the left-hand side is equal to $$-a_kB_k +2 B_k-B_k -B_{k-1} =-a_kB_k -B_{k-1}+ B_k =B_{k+1}+B_k,$$ where the last equality follows from Definition~(\ref{eq: B}). Consequently, $$\dfrac{1}{B_{k+1}} \sum_{i=0}^{k} (a_i -2) (-B_{k+1} + (-1)^{k+1-i} B_i)=
  2k+3 -(a_0 +a_1+\cdots+a_k) + \dfrac{B_k}{B_{k+1}}. $$

To prove the subequation~(\ref{eqn:binomi}a), it suffices to show that 
\begin{equation} \label{eq: matr} 
 I_k \begin{bmatrix}
\frac{1}{2} (-1+(-1)^{k+1} B_0/B_{k+1}) \\
-1+(-1)^{k+1} B_0/B_{k+1}  \\
-1+(-1)^{k} B_1/B_{k+1} \\
.\\
 .\\
 .\\
-1+(-1)  B_k/B_{k+1} \\
\frac{1}{2} (-1+(-1)^{k+1} B_0/B_{k+1})
\end{bmatrix}  = \begin{bmatrix}
0 \\
a_0-2  \\
a_1-2 \\
.\\
 .\\
 .\\
a_k-2 \\
0 
\end{bmatrix}.
\end{equation}

To verify  Equation~(\ref{eq: matr}) for $k=0$, we need to show that 
$$ \begin{bmatrix}
-2 & 1 & 0 \\
1 & -a_0 & 1 \\
 0 & 1 & -2
\end{bmatrix} \begin{bmatrix}
\frac{1}{2} (-1- B_0/B_{1}) \\
-1- B_0/B_{1} \\
\frac{1}{2} (-1- B_0/B_{1}) = \end{bmatrix} = \begin{bmatrix}
0 \\
a_0-2 \\
0
\end{bmatrix}$$ 
and it is clear that only the middle entry needs some verification: $$(1-a_0)(-1 -\dfrac{B_0}{B_1}) = (1-a_0)(-1 +\dfrac{1}{a_0-1})= a_0-2,$$ where we used the fact that $B_1= -(a_0-1)$ for the first equality.

To verify  Equation~(\ref{eq: matr}) for $k=1$, we need to show that $$ \begin{bmatrix}
-2 & 1 & 0 & 0\\
1 & -a_0 & 1 & 1\\
0& 1 & -a_1 &  0\\
 0 & 1 & 0 &  -2
\end{bmatrix} \begin{bmatrix}
\frac{1}{2} (-1+B_0/B_{2}) \\
-1+B_0/B_{2} \\
-1-B_1/B_{2} \\
\frac{1}{2} (-1+B_0/B_{2}) = \end{bmatrix} = \begin{bmatrix}
0 \\
a_0-2 \\
a_1-2 \\
0
\end{bmatrix}$$ and again it is clear that only the second and third entries need some verification. To verify the second entry we calculate that  
$$(1-a_0) (-1+\dfrac{B_0}{B_2}) -1 - \dfrac{B_1}{B_2} =a_0-1+  \dfrac{B_1}{B_2}  -1 - \dfrac{B_1}{B_2}=a_0-2,$$ where we used the fact that $B_1= -(a_0-1)$ for the first equality and to verify the third  entry we calculate that 
$$-1+\dfrac{B_0}{B_2} - a_1 (-1-\dfrac{B_1}{B_2}) = -1 +a_1 + \dfrac{a_1 B_1 + B_0}{B_2} = a_1-2,$$ where we used the fact that $B_2= -(a_1B_1 + B_0)$ for the last equality.

In general,  to verify  Equation~(\ref{eq: matr}) for an arbitrary integer $k\geq 2$, we proceed as follows.  It is clear that the first and last entries of the vector on the right-hand side of Equation~(\ref{eq: matr}) are both zero because
$$ -2 [\frac{1}{2} (-1+(-1)^{k+1} B_0/B_{k+1})] + [-1+(-1)^{k+1} B_0/B_{k+1}]=0.$$ 

To verify that the second entry of the vector on the right-hand side of Equation~(\ref{eq: matr}), we just calculate 
$$(1-a_0)(-1+(-1)^{k+1} \dfrac{B_0}{B_{k+1}})+(-1+(-1)^{k} \dfrac{B_1}{B_{k+1}})= a_0-2 + \dfrac{(-1)^{k+1}}{B_{k+1}}((1-a_0)B_0  -B_1)),$$ which is equal to $a_0-2$ since $B_0=1$ and $B_1=1-a_0$ by  Definition~(\ref{eq: B}). 

For $i=3, 4, \ldots, k+1$, to verify the $i$th   entry of the vector on the right-hand side of Equation~(\ref{eq: matr}),  we just need to calculate the expression
$$(-1+(-1)^{j+1} \dfrac{B_{i-3}}{B_{k+1}}) - a_{i-2} (-1+(-1)^{j} \dfrac{B_{i-2}}{B_{k+1}})+(-1+(-1)^{j-1} \dfrac{B_{i-1}}{B_{k+1}}),$$ where $j=k+3-i$, which is equal to 
$$a_{i-2} -2 + \dfrac{(-1)^{j+1}}{B_{k+1}} (B_{i-3}+a_{i-2}B_{i-2}+B_{i-1}).$$ But note that $B_{i-3}+a_{i-2}B_{i-2}+B_{i-1}=0$, by  Definition~(\ref{eq: B}), and hence the result is just $a_{i-2} -2$.

Finally, to verify the $(k+2)$nd entry of the vector on the right-hand side of Equation~(\ref{eq: matr}),  we calculate 
$$ -1 +\dfrac{B_{k-1}}{B_{k+1}} -a_{k} (-1-\dfrac{B_{k}}{B_{k+1}})= -1 + a_{k}+\dfrac{a_{k}B_{k} + B_{k-1}}{B_{k+1}}=  a_{k} -2,$$ where we used the fact that $B_{k+1}= -(a_{k}B_{k} + B_{k-1})$ for the last equality.
\end{proof}

\begin{proposition} \label{prop: theta} For the canonical contact structure $\xi_{can}$ on $D(p,q)$,   we have 
\begin{equation}\label{eq: theta}
\theta (\xi_{can}) =1- I(p/q)-\dfrac{1}{[a_k, a_{k-1}, \ldots, a_1, a_0-1]}.
\end{equation}
\end{proposition}

\begin{proof} By Lemma~\ref{lem: theta} and  Definition~\ref{def: i}, we have 
\begin{equation} \label{eq: gompf}
\begin{split}
\theta (\xi_{can}) &= c_1^2 (X(p,q), J_{\xi_{can}}) - 3\sigma(X(p,q))-2\chi (X(p,q)) \\
 & = 2k+3-(a_0 +a_1+\cdots+a_k) - \dfrac{1}{[a_k, \ldots, a_1, a_0-1]} +3(k+3) -2(k+4) \\
  & =  1+ 3k+3 -(a_0 +a_1+\cdots+a_k) - \dfrac{1}{[a_k, \ldots, a_1, a_0-1]} \\
   & = 1- I(p/q)-\dfrac{1}{[a_k, a_{k-1}, \ldots, a_1, a_0-1]}.
\end{split}
\end{equation}
 \end{proof} 

\begin{lemma} \label{lem: min} The canonical contact structure $\xi_{can}$ on $D(p,q)$,  satisfies
$$-2 < \theta (\xi_{can}),$$ provided that  $(p-q)/q' \in \mathcal{R}$, where $0 < q' < p - q$ is
the reduction of $q$ modulo $p - q$.
\end{lemma}

\begin{proof} Suppose that  $(p,q)$ is a coprime pair of integers such that $0 < q < p$ and let
$p/q=[a_0, a_1, \ldots, a_k]$ be the continued fraction expansion, where $a_i \geq 2$ for all $0 \leq i \leq k$.  

First we assume that $p-q > q$, which is equivalent to assuming that $a_0 > 2$. The spherical $3$-manifold $D(p,q)$ is rational homology cobordant to the lens space $L(p-q, q)$, which can be described by the linear plumbing graph with weights  $-(a_0-1), -a_1, \ldots, -a_k$, see Remark~\ref{rem: cond}. By \cite{Lisca2007} the condition $(p-q)/q \in\mathcal{R}$ is equivalent to the condition that $L(p-q, q)$ bounds a rational homology ball.  Assuming that $(p-q)/q \in\mathcal{R}$, we conclude that the inequality $I ((p-q)/q) \leq 1$ is satisfied by Lisca's work \cite{Lisca2007} (see also, \cite[Appendix~A.2]{AcetoGollaLarsonLecuona2020pre}), and thus $I (p/q) \leq 2$ since $I((p-q)/q)=I(p/q)-1$. As a consequence, we obtain that $-1 \leq 1-I(p/q)$.  Note that $1 < [a_k, a_{k-1}, \ldots, a_1, a_0-1],$ unless $a_i =2$ for all $0 \leq i \leq k$, but since we assumed that $a_0 > 2$, the inequality  $-2 < \theta (\xi_{can})$ follows immediately from Formula~\ref{eq: theta}.

Next we assume that $p-q < q$, which is equivalent to assuming that $a_0 = 2$. If  $a_i =2$, for all $0 \leq i \leq k$, then one may easily compute that $-2 <k+1= \theta (\xi_{can})$. Otherwise,  there exists $0 \leq l < k$ such that $a_i =2$, for all $0 \leq i \leq l$, but $a_{l+1} > 2$. Then $D(p,q)$ is rational homology cobordant to $L(p-q, q')$, which can be described by the linear plumbing graph with weights  $-(a_{l+1}-1), -a_{l+2}, \ldots, -a_k$, by successively blowing down the $(-1)$-curves.  But the string $((a_{l+1}-1), a_{l+2}, \ldots, a_k)$ or equivalently $(p-q)/q'$ is constrained as above so that  $I ((p-q)/q') \leq 1$. It follows that $$ 0 \leq l \leq 1+ l -I((p-q)/q') =1-I(p/q),$$ and thus we obtain  $-1 \leq l-1 < \theta (\xi_{can})$ again by Formula~\ref{eq: theta}.
\end{proof}

\begin{remark}
The converse of Lemma~\ref{lem: min} does not hold. For example, we have $-2 < -3/8 = \theta (\xi_{can})$ for $D(11,3)$,  but $8/3 \notin \mathcal{R}$.  
\end{remark} 

We can now rule out symplectic rational homology ball fillings of the $D(p,q)$. 
\begin{proposition}\label{prop2}
No contact structure on $D(p,q)$ is symplectically filled by a rational homology ball.
\end{proposition}
\begin{proof}
If $(p-q)/q'\in \mathcal{R}$ then Lemma~\ref{lem: min} tells us $\theta(\xi_{can})>-2$ and  Proposition~\ref{canminimizes} then tells us that $\theta(\xi)>-2$ for any tight contact structure $\xi$. Since none of these contact structures have $\theta=-2$ they cannot be symplectically filled by a rational homology ball by Lemma~\ref{lem: rhb}. If $(p-q)/q'\not\in \mathcal{R}$ then we know that $D(p,q)$ does not even smoothly bound a rational homology ball by Theorem~\ref{thm: ball}. 
\end{proof}

\subsection{The case of $T_3$, $T_{27}$,  $I_{49}$} 
We begin with a simple observation. 
\begin{lemma} \label{lem: item 3}  The canonical contact structure $\xi_{can}$ on $T_3$, $T_{27}$, and $I_{49}$,  satisfies
$$-2 < \theta (\xi_{can}).$$  \end{lemma}

\begin{proof} We simply compute that $\theta (\xi_{can})$ is equal to $22/9$,   $-122/81$,  and $-18/49$, respectively,  for $T_3$, $T_{27}$, and $I_{49}$.
\end{proof}
We can now rule out symplectic rational homology ball fillings of $T_3$, $T_{27}$, and $I_{49}$.
\begin{proposition}\label{prop3}
No contact structure on  $T_3$, $T_{27}$, or $I_{49}$ can be symplectically filled by a rational homology ball. 
\end{proposition}
\begin{proof}
Lemma~\ref{lem: rhb} and Lemma~\ref{lem: item 3} rule out $\xi_{can}$ having a symplectic rational homology ball filling and the other contact structure can be ruled out using Proposition~\ref{canminimizes}. (We can alternatively rule out such fillings using Theorem~\ref{main1}). 
\end{proof}
\subsection{Symplectic rational homology ball fillings of spherical $3$-manifolds} 
In this section we complete the proof of Theorem~\ref{thm: spherical}. We recall that the theorem says that if a contact structure $\xi$ on the spherical $3$-manifold $Y$, oriented as the link of the corresponding quotient singularity, admits a symplectic rational homology ball filling,  then $Y$ is orientation-preserving diffeomorphic to a lens space $L(p,q)$ with $p/q \in \mathcal{O}$, and $\xi$ is contactomorphic to $\xi_{can}$. 
\begin{proof}[Proof of Theorem~\ref{thm: spherical}]
The theorem is a direct consequence of Theorem~\ref{thm: ball} and Propositions~\ref{prop1}, \ref{prop2}, and~\ref{prop3}. 
\end{proof} 

\subsection{A proof of Proposition~\ref{prop: thetalens}}  \label{subsec: thetalens}
In this section we give a proof of  Proposition~\ref{prop: thetalens}, which says that $ \theta (\xi_{can}) = -\left(I(p/q) + (2+q+q^*)/p \right)$ for the canonical contact structure $\xi_{can}$ on $L(p,q)$,  where $0 < q^* < p$ is the multiplicative inverse of $q$ mod $p$.  

\begin{proof}[Proof of Proposition~\ref{prop: thetalens}] Suppose that  $(p,q)$ is a coprime pair of integers such that $0 < q < p$ and let
$p/q=[a_0, a_1, \ldots, a_k]$ be the continued fraction expansion, where $a_i \geq 2$ for all $0 \leq i \leq k$.  
It is easy to check that the lens space $L(p-q, q)$ can be described by the linear plumbing graph with weights  $-(a_0-1), -a_1, \ldots, -a_k$, where we assume $a_0 > 2$; otherwise one can blow-down the $(-1)$-curve.  

The $(k+1) \times (k+1)$ intersection matrix $Q_{L}$ of the plumbed $4$-manifold $W(p-q,q)$ with boundary $L(p-q,q)$  can be given as 
$$Q_{L}=\left[\begin{array}{cccccc}
 -(a_0-1) & 1 & &    \\
 1 & -a_1 & 1& &  \\ 
 & 1 & -a_2 & \ddots &  \\
 & & \ddots & \ddots & 1 \\
 & & & 1 & -a_k  \\
\end{array}\right], $$ whereas the $(k+3) \times (k+3)$ intersection matrix  $Q_{D}$ of the plumbed $4$-manifold $X(p,q)$ with boundary $D(p,q)$ can be given as  
$$Q_{D}=\left[\begin{array}{cccccc}
-2 & 1 & & & & \\
1 & -a_0 & 1 & &  & 1 \\
& 1 & -a_1 & \ddots & & \\
& & \ddots & \ddots & 1 \\
& & & 1 & -a_k & \\
& 1 & & & & -2
\end{array}\right]. $$ 

We claim that $\det Q_D = 4 \det  Q_L$. To prove our claim, first we add the last row of $Q_D$ to its second row, and observe that 
$$   \det Q_D = \begin{vmatrix}
-2 & 1 \\
1 & -a_0 & 1 & &  & 1 \\
& 1 & -a_1 & \ddots & & \\
& & \ddots & \ddots & 1 \\
& & & 1 & -a_{k} & \\
& 1 & & & & -2
\end{vmatrix} =  \begin{vmatrix}
-2 & 1 \\
1 & -(a_0-1)& 1 & &  &  -1 \\
& 1 & -a_1 & \ddots & & \\
& & \ddots & \ddots & 1 \\
& & & 1 & -a_{k} & \\
& 1 & & & &  -2
\end{vmatrix}. $$ Now we expand the determinant on the right hand side along its top row, to obtain 
$$ \det Q_D = -2  \begin{vmatrix}
 -(a_0-1)& 1 & &  &  -1 \\
 1 & -a_1 & \ddots & & \\
 & \ddots & \ddots & 1 \\
 & & 1 & -a_{k} & \\
 1 & & & &  -2
\end{vmatrix} -  \begin{vmatrix}
1 &  1 & &  & & -1 \\
& -a_1 & 1 & & \\
 &  & \ddots & \ddots \\
 &&& & 1\\
 && & 1 & -a_{k} & \\
&&  & & &   -2
\end{vmatrix} $$ and finally expanding both of these determinants along their bottom rows we get three determinants, one of which is 
$$ 4 \begin{vmatrix}
 -(a_0-1) & 1 & &    \\
 1 & -a_1 & 1& &  \\ 
 & 1 & -a_2 & \ddots &  \\
 & & \ddots & \ddots & 1 \\
 & & & 1 & -a_k 
\end{vmatrix} = 4 \det Q_L$$ and the sum of the remaining two is given as follows:  
$$-2 (-1)^{k+2} \begin{vmatrix}
  1 & &  &  -1 \\
  -a_1 & \ddots & & \\
  & \ddots & 1 \\
  & 1 & -a_{k} & 
\end{vmatrix} +2  \begin{vmatrix}
1 &  1 & &  &  \\
& -a_1 & 1 &  \\
 &  & \ddots & \ddots \\
 &&& & 1\\
 && & 1 & -a_{k}  
\end{vmatrix}.$$ By expanding the first determinant along the last column and the second determinant along the first column, we see that the sum above is equal to $$\underbrace{(-2)(-1)^{k+2} (-1) (-1)^{k+1}}_{-2}  \begin{vmatrix}
 -a_1 & 1 &  \\
   & \ddots & \ddots \\
 && & 1\\
 & & 1 & -a_{k}  
\end{vmatrix} +2  \begin{vmatrix}
 -a_1 & 1 &  \\
   & \ddots & \ddots \\
 && & 1\\
 & & 1 & -a_{k}  
\end{vmatrix} =0$$ proving our claim that $\det Q_D = 4 \det  Q_L$. 

In the following, for any square matrix $M$, we use $M(i,j)$ to denote its entry at the $i$th row and $j$th column and  $M_{i,j}$ to denote the matrix obtained from $M$ by deleting its $i$th row and $j$th column. 

Next we claim that the  $(k+1) \times (k+1)$ central submatrix of $Q^{-1}_{D}$ is equal to the matrix $Q^{-1}_{L}$. To prove this claim,  suppose that $3 \leq i\leq k+2$ and $3 \leq j \leq k+2$ (the case $i=2$ or $j=2$ is easier and we will discuss these cases separately at the end of our proof). To calculate $Q^{-1}_{D} (i,j)$ we need to calculate $\det (Q_D)_{i,j}$ and to do so we first add the last row of $(Q_D)_{i,j}$ to its second row and call the new matrix $(Q_D)'_{i,j}$. Then expanding along the top row of $(Q_D)'_{i,j}$ we obtain 
$$\det (Q_D)_{i,j} = \det (Q_D)'_{i,j} = -2 \det ((Q_D)'_{i,j})_{1,1} - \det ((Q_D)'_{i,j})_{1,2}$$  and by expanding both of these determinants along the last row, we get three determinants one of which is $$4 \det (((Q_D)'_{i,j})_{1,1})_{k+1, k+1} = 4 \det (Q_L)_{i-1, j-1} $$ and the remaining two cancel out  each other. Hence we showed that $$\det (Q_D)_{i,j} = 4 \det (Q_L)_{i-1, j-1} $$ which proves that $Q^{-1}_{D} (i,j) = Q^{-1}_{L} (i-1,j-1)$, in the light of the fact that $\det Q_D = 4 \det  Q_L$, provided that  $i \neq 2$ and $j \neq 2$. Finally, assuming that $j=2$ (the case $i=2$ is similar), we have $$ \det (Q_D)_{i,2} = -2  \det ((Q_D)_{i,2})_{1,1} = 4 \det (((Q_D)_{i,2})_{1,1})_{k+1, k+1}  = 4 \det (Q_L)_{i-1, 1} $$ which again proves that $Q^{-1}_{D} (i,2) = Q^{-1}_{L} (i-1,1)$.

This simple observation that  the  $(k+1) \times (k+1)$ central submatrix of $Q^{-1}_{D}$ is equal to the matrix $Q^{-1}_{L}$ can now be used to derive a formula for $c_1^2 (W(p-q,q), J_{\xi_{can}})$ using the formula $c_1^2 (X(p,q), J_{\xi_{can}})$ that we already obtained in  Lemma~\ref{lem: theta} as follows. Let $${\bf r}_L=   [a_0-3 \;\; a_1-2 \; \cdots \; \;a_k-2]^T  \in \mathbb{R}^{k+1}$$  denote the rotation vector for the canonical contact structure $\xi_{can}$ on $L(p-q,q)$ and set $$ {\bf \tilde{r}}_L =  [0 \;\; a_0-3 \;\; a_1-2 \; \cdots \; \;a_k-2\;\; 0]^T  \in \mathbb{R}^{k+3}.$$ By setting ${\bf e}=[0 \;\; 1\;\; 0\;\; \cdots \;\;0]^T \in \mathbb{R}^{k+3}$, the rotation vector for the canonical contact structure $\xi_{can}$ on $D(p,q)$ can be given by ${\bf r}_D = {\bf \tilde{r}}_L +{\bf e}$. Now, we have  

\begin{subequations}\label{eq: comp}
      \begin{align}
c_1^2 (X(p,q), J_{\xi_{can}}) &=  {\bf r}^T_D Q^{-1}_{D} {\bf r}_D \\
 & =  ({\bf \tilde{r}}_L +{\bf e})^T Q^{-1}_{D} ({\bf \tilde{r}}_L +{\bf e})  \\
  & =   {\bf \tilde{r}}_L ^T Q^{-1}_{D} {\bf \tilde{r}}_L  +{\bf \tilde{r}}_L ^T Q^{-1}_{D} {\bf e} +  {\bf e}^T Q^{-1}_{D} {\bf \tilde{r}}_L + {\bf e}^T Q^{-1}_{D} {\bf e}
      \end{align}
    \end{subequations}

Note that the first term ${\bf \tilde{r}}_L ^T Q^{-1}_{D} {\bf \tilde{r}}_L$ in ~(\ref{eq: comp}c) can be replaced with  ${\bf r}_L ^T Q^{-1}_{L} {\bf r}_L$ since  the  $(k+1) \times (k+1)$ central submatrix of $Q^{-1}_{D}$ is equal to the matrix $Q^{-1}_{L}$, and the first and last entries of  ${\bf \tilde{r}}_L $ are zeros. Moreover,   we know that $${\bf r}_L ^T Q^{-1}_{L} {\bf r}_L =  c_1^2 (W(p-q,q), J_{\xi_{can}}).$$  

The last term ${\bf e}^T Q^{-1}_{D} {\bf e}$  in ~(\ref{eq: comp}c) is exactly $Q^{-1}_{D} (2,2)= Q^{-1}_{L}(1,1)$ which we compute as follows. Note that $\det Q_L = (-1)^{k+1} (p-q) $ and $\det (Q_L)_{1,1} = (-1)^{k} q$. Therefore, the $(1,1)$-entry of $Q^{-1}_{L}$ is equal to $$Q^{-1}_{L}(1,1)=\dfrac{\det (Q_L)_{1,1}}{\det Q_L} = \dfrac{(-1)^{k} q}{(-1)^{k+1} (p-q) }=-\dfrac{q}{p-q}.$$ 

Finally, we compute the sum ${\bf \tilde{r}}_L ^T Q^{-1}_{D} {\bf e} +  {\bf e}^T Q^{-1}_{D} {\bf \tilde{r}}_L$ of the two middle terms  in ~(\ref{eq: comp}c). Note that $Q^{-1}_{D} {\bf e}$ is the second column of $Q^{-1}_{D}$ and hence  ${\bf \tilde{r}}_L ^T Q^{-1}_{D} {\bf e}$  is the dot product of ${\bf \tilde{r}}_L $ with the second column of $Q^{-1}_{D}$. Similarly,   ${\bf e}^T Q^{-1}_{D} {\bf \tilde{r}}_L$ is the dot product of ${\bf \tilde{r}}_L $ with the second row of $Q^{-1}_{D}$. But since $Q^{-1}_{D}$ is symmetric,  we conclude that the two terms in the sum $${\bf \tilde{r}}_L ^T Q^{-1}_{D} {\bf e} +  {\bf e}^T Q^{-1}_{D} {\bf \tilde{r}}_L$$ are actually the same and therefore their sum is equal to twice the dot product of ${\bf \tilde{r}}_L $ with the second row of $Q^{-1}_{D}$. To compute this dot product we make use of Equation~(\ref{eq: matr}), where $I_k$ there is denoted by $Q_D$ in this proof. According to Equation~(\ref{eq: matr}), 
$$Q^{-1}_D {\bf r}_D= \begin{bmatrix}
\frac{1}{2} (-1+(-1)^{k+1} B_0/B_{k+1}) \\
-1+(-1)^{k+1} B_0/B_{k+1}  \\
-1+(-1)^{k} B_1/B_{k+1} \\
.\\
 .\\
 .\\
-1+(-1)  B_k/B_{k+1} \\
\frac{1}{2} (-1+(-1)^{k+1} B_0/B_{k+1})
\end{bmatrix} $$ which implies that the dot product of the second row of $Q^{-1}_{D}$ with  ${\bf r}_D ={\bf \tilde{r}}_L + {\bf e} $ is equal to $$-1+(-1)^{k+1} B_0/B_{k+1} =-1+ (-1)^{k+1}\dfrac{1}{(-1)^{k+1} (p-q)} = -1 + \dfrac{1}{p-q}.$$ Thus to calculate the dot product of ${\bf \tilde{r}}_L $ with the second row of $Q^{-1}_{D}$ we just need to {\em subtract} from this expression the dot product of $\bf e$ with the second row of $Q^{-1}_{D}$, which is nothing but $$Q^{-1}_{D}(2,2)= Q^{-1}_{L}(1,1)=-\dfrac{q}{p-q}.$$ As a consequence, we  calculate that the sum of the two middle terms  in ~(\ref{eq: comp}c) as 
$$
2 \left(-1 + \dfrac{1}{p-q} + \dfrac{q}{p-q}\right).
$$

It follows that   
\begin{equation} \label{eq: compare} 
\begin{split}
c_1^2 (X(p,q), J_{\xi_{can}}) &= c_1^2 (W(p-q,q), J_{\xi_{can}})  + 2 \left(-1 + \dfrac{1}{p-q} + \dfrac{q}{p-q}\right)- \dfrac{q}{p-q} \\
 & = c_1^2 (W(p-q,q), J_{\xi_{can}}) -2 + \dfrac{2+q}{p-q}
\end{split}
\end{equation}

Let $\theta_{D(p,q)}$ denote the $\theta$-invariant of the canonical contact structure on $D(p,q)$ and similarly let $\theta_{L(p-q,q)}$ denote the $\theta$-invariant of the canonical contact structure on $L(p-q,q)$. An elementary calculation about the signatures and Euler characteristics of the $4$-manifolds $X(p,q)$ and $W(p-q,q)$, yields that $$ \theta_{D(p,q)}= c_1^2 (X(p,q), J_{\xi_{can}}) +k+1$$ and $$ \theta_{L(p-q,q)}= c_1^2 (W(p-q,q), J_{\xi_{can}}) +k-1.$$ Combining with  Equation~(\ref{eq: compare}), we conclude that 
$$\theta_{L(p-q,q)} = \theta_{D(p,q)} -  \dfrac{2+q}{p-q},$$  which in turn, leads to $$\theta_{L(p-q,q)} = 1- I(p/q)-\dfrac{1}{[a_k, a_{k-1}, \ldots, a_1, a_0-1]} -  \dfrac{2+q}{p-q}$$ by Proposition~\ref{prop: theta}. By substituting $\widetilde{p}$ for $p-q$, we obtain $$\theta_{L(\widetilde{p},q)} = 1- I(\widetilde{p}/q)-1- \dfrac{q*}{\widetilde{p}} -  \dfrac{2+q}{\widetilde{p}} = -I(\widetilde{p}/q) -  \dfrac{2+q+q*}{\widetilde{p}} $$ which proves the desired result, by replacing $\widetilde{p}$ by $p$ and noticing that the notation $\theta_{L(p,q)}$ is used in our proof to denote  $\theta(\xi_{can})$ for the canonical contact structure $\xi_{can}$ on $L(p,q)$. 
\end{proof}

\bibliography{references}
\bibliographystyle{plain}

\end{document}